\documentclass[11pt]{amsart}

\usepackage{
    amsmath,
    amsfonts,
    amssymb,
    amsthm,
    amscd,
    comment,
    enumitem,
    etoolbox,
    gensymb,    % For \degree
    mathtools,
    mathdots,
    stmaryrd
}
\usepackage[usenames,dvipsnames]{xcolor}
\usepackage[all]{xy}

\usepackage[T1]{fontenc}
\usepackage{bbm}                     % For \mathbbm{1} (unit in monoidal category)
\usepackage[colorlinks=true, linkcolor=blue, citecolor=blue, urlcolor=blue, breaklinks=true]{hyperref}

\DeclareFontFamily{OT1}{pzc}{}
\DeclareFontShape{OT1}{pzc}{m}{it}{<-> s * [1.10] pzcmi7t}{}
\DeclareMathAlphabet{\mathpzc}{OT1}{pzc}{m}{it}

\usepackage[capitalize]{cleveref}   % Can add 'nameinlink' option to have the name included in hyperlink

\crefname{defin}{Definition}{Definitions}
\crefname{eg}{Example}{Examples}
\crefname{egs}{Example}{Examples}
\crefname{lem}{Lemma}{Lemmas}
\crefname{theo}{Theorem}{Theorems}
\crefname{equation}{}{}
\crefname{enumi}{}{}
\newcommand\N{\mathbb{N}}

\newcommand\Z{\mathbb{Z}}
\newcommand\kk{\Bbbk}
\newcommand\one{\mathbbm{1}}

\newcommand\B{\mathbf{B}}

\newcommand\cA{\mathcal{A}}
\newcommand\cB{\mathcal{B}}
\newcommand\cC{\mathcal{C}}
\newcommand\cD{\mathcal{D}}

\newcommand\fg{\mathfrak{g}}
\newcommand\gl{\mathfrak{gl}}          % General linear Lie algebra
\newcommand\Cl{\mathrm{Cl}}

\newcommand\rmd{\textup{mod-}}
\newcommand\smd{\textup{smod-}}

\newcommand\even{{\bar{0}}}
\newcommand\odd{{\bar{1}}}
\newcommand\dA{d_A}                     % Top degree of A
\newcommand\cocenter[1]{%   For element in the cocenter
    \dot{#1}
}

\newcommand\AOB{\mathpzc{AOB}}          % Affine oriented Brauer category
\newcommand\cEnd{\mathpzc{End}}
\newcommand\Heis{\mathpzc{Heis}}        % Heisenberg category
\newcommand\OB{\mathpzc{OB}}            % Oriented Brauer category
\newcommand\SVec{\mathpzc{SVec}}        % Vector superspaces

\DeclareMathOperator{\End}{End}

\DeclareMathOperator{\Hom}{Hom}

\DeclareMathOperator{\Mat}{Mat}

\DeclareMathOperator{\sdim}{sdim}   % Super dimension

\DeclareMathOperator{\str}{str}     % supertrace
\DeclareMathOperator{\Sym}{Sym}

\DeclareMathOperator{\tr}{tr}

\usepackage{tikz}
\usetikzlibrary{arrows.meta}
\usetikzlibrary{decorations.markings}
\usetikzlibrary{calc}

\tikzset{anchorbase/.style={>=To,baseline={([yshift=-0.5ex]current bounding box.center)}}}
\tikzset{ % Syntax: \begin{tikzpicture}[centerzero={0,0.2}]
    centerzero/.style={>=To,baseline={([yshift=-0.5ex](#1))}},
    centerzero/.default={0,0}
}
\tikzset{wipe/.style={white,line width=3pt}}

\newcommand\braidup{to[out=up,in=down]}
\newcommand\braiddown{to[out=down,in=up]}

\newcommand\dotlabel[1]{$\scriptstyle{#1}$}
\newcommand\token[3]{%  \token{position}{anchor}{label}
    \filldraw[blue] (#1) circle (1.5pt) node[anchor=#2] {\dotlabel{#3}}
}
\newcommand\singdot[1]{% \singdot{position}
    \filldraw[fill=white, draw=red] (#1) circle (1.5pt)
}
\newcommand\multdot[3]{% \multdot{position}{anchor}{label}
    \filldraw[fill=white, draw=red] (#1) circle (1.5pt) node[anchor=#2] {{\color{red} \dotlabel{#3}}}
}
\newcommand\teleport[2]{% \teleport{position1}{position2}
    \draw (#1) -- (#2);
    \filldraw[blue] (#1) circle (1.5pt);
    \filldraw[blue] (#2) circle (1.5pt);
}

\newcommand\bubrightblank[1]{% \bubrightblank{position}
    \draw[->] (#1)++(0,0.2) arc(90:-270:0.2)
}
\newcommand\bubright[3]{% \bubright{position}{token}{dot}
    \draw[->] (#1)++(0,0.2) arc(90:-270:0.2);
    \filldraw[blue] (#1)++(0.2,0) circle (1.5pt) node[anchor=west] {\dotlabel{#2}};
    \filldraw[fill=white, draw=red] (#1)++(-0.2,0) circle (1.5pt) node[anchor=east] {{\color{red} \dotlabel{#3}}}
}
\newcommand\bubleftblank[1]{% \bubleftblank{position}
    \draw[->] (#1)++(0,0.2) arc(90:450:0.2)
}
\newcommand\bubleft[3]{% \bubleft{position}{token}{dot}
    \draw[->] (#1)++(0,0.2) arc(90:450:0.2);
    \filldraw[blue] (#1)++(-0.2,0) circle (1.5pt) node[anchor=east] {\dotlabel{#2}};
    \filldraw[fill=white, draw=red] (#1)++(0.2,0) circle (1.5pt) node[anchor=west] {{\color{red} \dotlabel{#3}}}
}

\newcommand\cbubble[2]{% \cbubble{token}{dot}
    \begin{tikzpicture}[centerzero]
        \bubright{0,0}{#1}{#2};
    \end{tikzpicture}
}
\newcommand\ccbubble[2]{% \ccbubble{token}{dot}
    \begin{tikzpicture}[centerzero]
        \bubleft{0,0}{#1}{#2};
    \end{tikzpicture}
}
\newcommand\uptok[1][a]{
    \begin{tikzpicture}[centerzero]
        \draw[->] (0,-0.2) -- (0,0.2);
        \token{0,0}{west}{#1};
    \end{tikzpicture}
}

\newcommand\updot{
    \begin{tikzpicture}[centerzero]
        \draw[->] (0,-0.2) -- (0,0.2);
        \singdot{0,0};
    \end{tikzpicture}
}
\newcommand\upcross{
    \begin{tikzpicture}[centerzero]
        \draw[->] (0.2,-0.2) -- (-0.2,0.2);
        \draw[->] (-0.2,-0.2) -- (0.2,0.2);
    \end{tikzpicture}
}
\newcommand\downcross{
    \begin{tikzpicture}[centerzero]
        \draw[<-] (0.2,-0.2) -- (-0.2,0.2);
        \draw[<-] (-0.2,-0.2) -- (0.2,0.2);
    \end{tikzpicture}
}
\newcommand\rightcross{
    \begin{tikzpicture}[centerzero]
        \draw[->] (-0.2,-0.2) -- (0.2,0.2);
        \draw[<-] (0.2,-0.2) -- (-0.2,0.2);
    \end{tikzpicture}
}
\newcommand\leftcross{
    \begin{tikzpicture}[centerzero]
        \draw[<-] (-0.2,-0.2) -- (0.2,0.2);
        \draw[->] (0.2,-0.2) -- (-0.2,0.2);
    \end{tikzpicture}
}
\newcommand{\rightcup}{
    \begin{tikzpicture}[anchorbase]
        \draw[->] (-0.15,0.15) -- (-0.15,0) arc(180:360:0.15) -- (0.15,0.15);
    \end{tikzpicture}
}
\newcommand{\leftcup}{
    \begin{tikzpicture}[anchorbase]
        \draw[<-] (-0.15,0.15) -- (-0.15,0) arc(180:360:0.15) -- (0.15,0.15);
    \end{tikzpicture}
}
\newcommand{\rightcap}{
    \begin{tikzpicture}[anchorbase]
        \draw[->] (-0.15,-0.15) -- (-0.15,0) arc(180:0:0.15) -- (0.15,-0.15);
    \end{tikzpicture}
}
\newcommand{\leftcap}{
    \begin{tikzpicture}[anchorbase]
        \draw[<-] (-0.15,-0.15) -- (-0.15,0) arc(180:0:0.15) -- (0.15,-0.15);
    \end{tikzpicture}
}

\newtheorem{theo}{Theorem}[section]

\newtheorem{prop}[theo]{Proposition}
\newtheorem{lem}[theo]{Lemma}

\theoremstyle{definition}
\newtheorem{defin}[theo]{Definition}
\newtheorem{rem}[theo]{Remark}
\newtheorem{eg}[theo]{Example}

\numberwithin{equation}{section}
\allowdisplaybreaks

\setenumerate[1]{label=(\alph*)}          % Only need to reference enumerate items with \ref{}

\newtoggle{comments}
\newtoggle{details}
\newtoggle{detailsnote}

\iftoggle{comments}{%
  \newcommand{\acomments}[1]{
    \ \\
    {\color{red}
      \textbf{AS:} #1
    }
    \ \\
    }
  \newcommand{\alex}[1]{
    \ \\
    {\color{purple}
      \textbf{For Alex:} #1
    }
    \ \\
    }
}{%
  \newcommand{\acomments}[1]{}
  \newcommand{\alex}[1]{}
}

\iftoggle{details}{%
  \newcommand{\details}[1]{
      \ \\
      {\color{OliveGreen}
        \textbf{Details:} #1
      }
      \\
  }
}{%
  \newcommand{\details}[1]{}
}

\begin{document}
%===============

\title{Affine oriented Frobenius Brauer categories}

\author{Alexandra McSween}
\address[A.M.]{
  Department of Mathematics and Statistics \\
  University of Ottawa \\
  Ottawa, ON K1N 6N5, Canada
}
\email{amcsw087@uottawa.ca}

\author{Alistair Savage}
\address[A.S.]{
  Department of Mathematics and Statistics \\
  University of Ottawa \\
  Ottawa, ON K1N 6N5, Canada
}
\urladdr{\href{https://alistairsavage.ca}{alistairsavage.ca}, \textrm{\textit{ORCiD}:} \href{https://orcid.org/0000-0002-2859-0239}{orcid.org/0000-0002-2859-0239}}
\email{alistair.savage@uottawa.ca}

\begin{abstract}
    To any Frobenius superalgebra $A$ we associate an \emph{oriented Frobenius Brauer category} and an \emph{affine oriented Frobenius Brauer category}.  We define natural actions of these categories on categories of supermodules for general linear Lie superalgebras $\mathfrak{gl}_{m|n}(A)$ with entries in $A$.  These actions generalize those on module categories for general linear Lie superalgebras and queer Lie superalgebras, which correspond to the cases where $A$ is the ground field and the two-dimensional Clifford superalgebra, respectively.
\end{abstract}

\subjclass[2020]{18M05, 18M30, 17B10}

\keywords{Monoidal category, supercategory, Lie superalgebra, Brauer algebra}

\ifboolexpr{togl{comments} or togl{details}}{%
  {\color{magenta}DETAILS OR COMMENTS ON}
}{%
}

\maketitle
\thispagestyle{empty}

%\tableofcontents

%=====================
\section{Introduction}
%=====================

The oriented Brauer category $\OB$ is the free linear rigid symmetric monoidal category generated by a single object $\uparrow$.  This universal property immediately implies the existence of a monoidal functor
\[
    F \colon \OB \to \rmd\gl_n
\]
from $\OB$ to the category of right modules for the general linear Lie algebra $\gl_n$.  (One can also work with left modules, but right modules turn out to be easier for the current paper; see \cref{vision}.)  This functor sends the generating object $\uparrow$ and its dual $\downarrow$ to the defining $\gl_n$-module $V$ and its dual $V^*$, respectively.  For $r \ge 1$, the endomorphism algebra $\End_{\OB}(\uparrow^{\otimes r})$ is the group algebra of the symmetric group on $r$ letters, and the algebra homomorphism
\[
    \End_\OB(\uparrow^{\otimes r}) \to \End_{\gl_n}(V^{\otimes r})
\]
induced by $F$ is the classical one appearing in Schur--Weyl duality.  More generally, $\End_\OB(\uparrow^{\otimes r} \otimes \downarrow^{\otimes s})$ are \emph{walled Brauer algebras} and the induced algebra homomorphisms
\[
    \End_\OB(\uparrow^{\otimes r} \otimes \downarrow^{\otimes s}) \to \End_{\gl_n}(V^{\otimes r} \otimes (V^*)^{\otimes s})
\]
were originally defined and studied by Turaev \cite{Tur89} and Koike \cite{Koi89}.

The rank $n$ of $\gl_n$ appears as a parameter in $\OB$.  In fact, the definition of $\OB$ makes sense for \emph{any} value of this parameter, i.e.\ it need not be a positive integer.  This observation leads to the definition of Deligne's interpolating category for the general linear Lie groups \cite{Del07}; this interpolating category is the additive Karoubi envelope of $\OB$.

The functor $F$ yields an action of $\OB$ on $\rmd\gl_n$.  More precisely, for a category $\cC$, let $\cEnd(\cC)$ denote the corresponding strict monoidal category of endofunctors and natural transformations.  Then we have a functor
\[
    \OB \to \cEnd(\rmd\gl_n),\quad
    X \mapsto F(X) \otimes -,\quad
    f \mapsto F(f) \otimes -,
\]
for objects $X$ and morphisms $f$ in $\OB$.  In \cite{BCNR17}, this was extended to an action of the \emph{affine oriented Brauer category} $\AOB$ on $\rmd\gl_n$.  The category $\AOB$ is obtained from $\OB$ by adjoining an additional endomorphism of the generating object $\uparrow$, subject to certain natural relations.  This additional endomorphism acts by a natural transformation of the functor $V \otimes -$ arising from multiplication by a certain canonical element of $\gl_n \otimes \gl_n$.  Restricting to endomorphism spaces recovers actions of \emph{affine walled Brauer algebras} studied in \cite{RS15,Sar14}.

In fact, much of the above picture can be generalized, replacing $\gl_n$ by the general linear Lie \emph{super}algebra $\gl_{m|n}$.  Remarkably, one does not need to modify $\OB$ or $\AOB$ at all.  Here the corresponding Schur--Weyl duality was established by Sergeev \cite{Ser84} and Berele--Regev \cite{BR87}, while the action of the walled Brauer algebras was described in \cite[Prop.~3.2]{LSM02} and \cite[Th.~7.8]{BS12}.  The analogue of the functor $F$ above is described in \cite[\S8.3]{CW12} and \cite[Th.~4.16]{ES16}.  The extension of the action to $\AOB$ does not seem to have appeared in the literature, but is certainly expected by experts.  For example, it is mentioned in the introduction to \cite{BCNR17}.

The affine oriented Brauer category $\AOB$ is a special case of more general category.  The \emph{Heisenberg category} at central charge $-1$ was first introduced by Khovanov \cite{Kho14} as a tool to study the representation theory of the symmetric group.  In \cite{MS18}, it was generalized to arbitrary negative central charge, which corresponds to replacing the symmetric group by more general degenerate cyclotomic Hecke algebras of type $A$.  In \cite{Bru18}, Brundan gave a simplified presentation of the Heisenberg category $\Heis_k$ at arbitrary central charge $k$.  When $k=0$, the Heisenberg category is precisely the affine oriented Brauer category.

The Heisenberg category has been further generalized in \cite{RS17,Sav19,BSW-foundations} to the \emph{Frobenius Heisenberg category} $\Heis_k(A)$ depending on a Frobenius superalgebra $A$.  When $A=\kk$, this construction recovers the Heisenberg category.  When $k \ne 0$, the category $\Heis_k(A)$ acts naturally on categories of modules over the cyclotomic wreath product algebras defined in \cite{Sav20}.  However, actions in the case $k=0$ have not yet been studied.  In some sense, central charge zero yields the simplest and most interesting case.  For example, $\Heis_k(A)$ is symmetric monoidal if and only if $k=0$.

The goal of the current paper is to fill in this gap in the literature.  In analogy with the $A=\kk$ case, we call the Frobenius Heisenberg category at central charge zero the \emph{affine oriented Frobenius Brauer category} $\AOB(A)$.  It contains a natural Frobenius algebra analogue of the oriented Brauer category, which we call the \emph{oriented Frobenius Brauer category} $\OB(A)$.  We define, in \cref{tiger,lion}, natural superfunctors
\[
    \OB(A) \to \smd\gl_{m|n}(A),\quad
    \AOB(A) \to \cEnd(\smd\gl_{m|n}(A)),
\]
where $\smd\gl_{m|n}(A)$ denotes the monoidal supercategory of right supermodules for the general linear Lie superalgebra with entries in the Frobenius superalgebra $A$.

When $A=\kk$, we recover the functors described above for $\gl_{m|n} = \gl_{m|n}(\kk)$.  On the other hand, if $A=\Cl$ is the two-dimensional Clifford superalgebra, then $\gl_{m|n}(\Cl)$ is isomorphic to the queer Lie superalgebra $\mathfrak{q}_{m+n}(\kk)$, and our superfunctors recover those defined in \cite{BCK19}.  As in the $A=\kk$ case, these superfunctors extend Schur--Weyl duality results for queer Lie superalgebras \cite{Ser84b}, actions of walled Brauer--Clifford superalgebras \cite{JK14}, and actions of affine walled Brauer--Clifford superalgebras \cite{BCK19,GRSS19}.  In fact, the Clifford superalgebra is the main example of interest where the Frobenius superalgebra is not symmetric.  Since this case has already been studied in the aforementioned papers, we assume throughout the current document that $A$ is symmetric, as this simplifies the exposition.  We have indicated in \cref{glass} the modification that needs to be made to handle the more general case.

The results of the current paper extend the powerful category theoretic tools that have been used to study the representation theory of general linear Lie superalgebras and queer Lie superalgebras to the setting of general linear Lie superalgebras over Frobenius superalgebras.  For example, in \cref{shark}, we see that these superfunctors yield central elements in the universal enveloping algebra generalizing the known generators of this center in the $A=\kk$ and $A=\Cl$ cases.    When $A=\kk[x]/(x^l)$, then $\gl_{m|n}(A)$ is a truncated current superalgebra (also called a Takiff algebra when $m=0$ or $n=0$).  In this case, Brauer category type methods do not seem to have appeared in the literature before.

The results of the current paper lead naturally to several avenues of further research.  We conclude this introduction by listing a few such directions here.
\begin{enumerate}
    \item \emph{Interpolating categories}.  The idempotent completion of $\OB(A)$ is a natural candidate for an interpolating category for $\smd\gl_{m|n}(A)$, which could potentially be used to generalize work of Deligne and others in the case $n=0$, $A=\kk$.

    \item \emph{Frobenius Schur algebras}.  One should be able to define Schur algebras depending on a Frobenius superalgebra $A$ such that, when $A=\kk$, one recovers the usual Schur algebras.

    \item \emph{Cyclotomic quotients}.  In the cases $A=\kk$ and $A=\Cl$, cyclotomic quotients of $\AOB(A)$ have been studied in \cite{BCNR17,BCK19}.  We expect that many of these results can be extended to the setting of general Frobenius superalgebras.

    \item \emph{Quantum analogues}.  The \emph{quantum Frobenius Heisenberg categories}, introduced in \cite{BSW-qFrobheis}, are natural quantum analogues of Frobenius Heisenberg categories.  The special case of central charge zero yields a natural \emph{quantum affine oriented Frobenius Brauer category}.  When $A=\kk$, this is the affine HOMFLY-PT skein category.  Then quantum affine oriented Frobenius Brauer categories should act on as-yet-to-be-defined quantum enveloping algebras of $\gl_{m|n}(A)$ and yield Frobenius analogues of the HOMFLY-PT link invariant.
\end{enumerate}

%-----------------------------
\subsection*{Acknowledgements}
%-----------------------------

This research was supported by Discovery Grant RGPIN-2017-03854 from the Natural Sciences and Engineering Research Council of Canada (NSERC).

%=================================
\section{Monoidal supercategories}
%=================================

Throughout the paper, we fix a ground field $\kk$.  All vector spaces, algebras, categories, and
functors will be assumed to be linear over $\kk$ unless otherwise specified.  Unadorned tensor products denote tensor products over $\kk$.  Most things in the article will be enriched over the category $\SVec$ of \emph{vector superspaces}, that is, $\Z/2\Z$-graded vector spaces $V = V_\even \oplus V_\odd$ with parity-preserving morphisms.  Writing $\bar v \in \Z/2\Z$ for the parity of a homogeneous vector $v \in V$, the category $\SVec$ is a symmetric monoidal category with symmetric braiding $V \otimes W \rightarrow W\otimes V$ defined by $v\otimes w \mapsto (-1)^{\bar v \bar w} w \otimes v$ for homogeneous $v,w$, and extended by linearity.

For superalgebras $A=A_\even \oplus A_\odd$ and $B = B_\even\oplus
B_\odd$, multiplication in the superalgebra $A \otimes B$ is defined
by \begin{equation}
(a' \otimes b) (a \otimes b') = (-1)^{\bar a \bar b} a'a \otimes
bb'
\end{equation}
for homogeneous $a,a' \in A$, $b,b' \in B$.  The \emph{center} $Z(A)$ is the subalgebra generated by all homogeneous $a \in A$ such that
\begin{equation}\label{center}
    ab = (-1)^{\bar a \bar b}ba
\end{equation}
for all homogeneous $b \in A$.  The \emph{cocenter} $C(A)$ is the quotient of $A$ by the subspace spanned by $ab-(-1)^{\bar a \bar b}ba$ for all homogeneous $a,b \in A$.  Note that, a priori, $C(A)$ is only a vector superspace, and not a superalgebra.

Throughout this document, we will be working with \emph{strict monoidal supercategories} in the sense of \cite{BE17}.  A \emph{supercategory} is a category enriched in $\SVec$.  Thus, its morphism spaces are superspaces and composition is parity preserving.  A \emph{superfunctor} between supercategories induces a parity-preserving linear map between morphism superspaces.  For superfunctors $F,G \colon \cA\rightarrow\cB$, a \emph{supernatural transformation} $\alpha \colon F \Rightarrow G$ of \emph{parity $r\in\Z/2\Z$} is the data of morphisms $\alpha_X\in \Hom_{\cB}(FX, GX)_r$ for each $X\in\cA$ such that $Gf \circ \alpha_X = (-1)^{r \bar f}\alpha_Y\circ Ff$ for each homogeneous $f \in \Hom_{\cA}(X, Y)$.  A \emph{supernatural transformation} $\alpha \colon F \Rightarrow G$ is $\alpha = \alpha_\even + \alpha_\odd$ with each $\alpha_r$ being a supernatural transformation of parity $r$.

In a strict monoidal supercategory, the \emph{super interchange law} is
\begin{equation}\label{interchange}
    (f' \otimes g) \circ (f \otimes g')
    = (-1)^{\bar f \bar g} (f' \circ f) \otimes (g \circ g').
\end{equation}
We denote the unit object by $\one$ and the identity morphism of an object $X$ by $1_X$.  We will use the usual calculus of string diagrams, representing the horizontal composition $f \otimes g$ (resp.\ vertical composition $f \circ g$) of morphisms $f$ and $g$ diagrammatically by drawing $f$ to the left of $g$ (resp.\ drawing $f$ above $g$).  Care is needed with horizontal levels in such diagrams due to
the signs arising from the super interchange law:
\begin{equation}\label{intlaw}
    \begin{tikzpicture}[anchorbase]
        \draw (-0.5,-0.5) -- (-0.5,0.5);
        \draw (0.5,-0.5) -- (0.5,0.5);
        \filldraw[fill=white,draw=black] (-0.5,0.15) circle (5pt);
        \filldraw[fill=white,draw=black] (0.5,-0.15) circle (5pt);
        \node at (-0.5,0.15) {$\scriptstyle{f}$};
        \node at (0.5,-0.15) {$\scriptstyle{g}$};
    \end{tikzpicture}
    \quad=\quad
    \begin{tikzpicture}[anchorbase]
        \draw (-0.5,-0.5) -- (-0.5,0.5);
        \draw (0.5,-0.5) -- (0.5,0.5);
        \filldraw[fill=white,draw=black] (-0.5,0) circle (5pt);
        \filldraw[fill=white,draw=black] (0.5,0) circle (5pt);
        \node at (-0.5,0) {$\scriptstyle{f}$};
        \node at (0.5,0) {$\scriptstyle{g}$};
    \end{tikzpicture}
    \quad=\quad
    (-1)^{\bar f\bar g}\
    \begin{tikzpicture}[anchorbase]
        \draw (-0.5,-0.5) -- (-0.5,0.5);
        \draw (0.5,-0.5) -- (0.5,0.5);
        \filldraw[fill=white,draw=black] (-0.5,-0.15) circle (5pt);
        \filldraw[fill=white,draw=black] (0.5,0.15) circle (5pt);
        \node at (-0.5,-0.15) {$\scriptstyle{f}$};
        \node at (0.5,0.15) {$\scriptstyle{g}$};
    \end{tikzpicture}
    \ .
\end{equation}
We refer the reader to \cite{Sav-exp} for a brief overview of string diagrams suited to the current paper, to \cite[Ch.~1, 2]{TV17} for a more in-depth treatment, and to \cite{BE17} for a detailed discussion of signs in the super setting.

%=======================================================================
\section{General linear Lie superalgebras over Frobenius superalgebras}
%=======================================================================

Throughout this document $A$ will denote a symmetric Frobenius superalgebra with parity-preserving trace map $\tr \colon A \to \kk$.  Thus
\begin{equation} \label{yield}
    \tr(ab) = (-1)^{\bar{a}\bar{b}} \tr(ba)
    = (-1)^{\bar{a}} \tr(ba)
    = (-1)^{\bar{b}} \tr(ba),
    \quad a,b \in A,
\end{equation}
where the second and third equalities follow from the fact that $\tr(ab)=0$ unless $\bar{a}=\bar{b}$.  The definition of a Frobenius superalgebra gives that $A$ possesses a homogeneous basis $\B_A$ and a left dual basis $\{b^\vee : b \in \B_A\}$ such that
\begin{equation}
    \tr(b^\vee c) = \delta_{b,c},\quad b,c \in \B_A.
\end{equation}
It follows that, for all $a \in A$, we have
\begin{equation} \label{adecomp}
    a
    = \sum_{b \in \B_A} \tr(b^\vee a)b
    = \sum_{b \in \B_A} \tr(ab) b^\vee.
\end{equation}
Note that $\bar{b} = \overline{b^\vee}$, and that the left dual basis to $\{b^\vee : b \in \B_A\}$ is given by
\begin{equation} \label{doubledual}
    (b^\vee)^\vee = (-1)^{\bar{b}} b.
\end{equation}

\begin{rem} \label{marvel}
    More generally, a (not necessarily symmetric) Frobenius superalgebra is an associative superalgebra with a trace map admitting dual bases, but where we do not require the trace map to satisfy \cref{yield}.  One can show that there is an automorphism $\varphi$ of $A$, called the \emph{Nakayama automorphism}, such that $\tr(ab) = (-1)^{\bar{a}\bar{b}} \tr(b \varphi(a))$.  More generally, one can also allow the trace map to be parity reversing.
\end{rem}

\begin{lem}
    For all homogeneous $a,c \in A$, we have
    \begin{equation} \label{beam}
        \sum_{b \in \B_A} (-1)^{\bar{b}\bar{c}} abc \otimes b^\vee
        = (-1)^{\bar{a}\bar{c}} \sum_{b \in \B_A} (-1)^{\bar{b}\bar{c}} b \otimes cb^\vee a.
    \end{equation}
\end{lem}

\begin{proof}
    We have
    \begin{multline*}
        \sum_{b \in \B_A} (-1)^{\bar{b}\bar{c}} abc \otimes b^\vee
        \overset{\cref{adecomp}}{=} \sum_{b,e \in \B_A} (-1)^{\bar{b}\bar{c}} \tr(e^\vee abc)e \otimes b^\vee
        = \sum_{b,e \in \B_A} (-1)^{\bar{c}\bar{e} + \bar{c}\bar{a}} e \otimes \tr(ce^\vee ab)b^\vee
        \\
        \overset{\cref{adecomp}}{=} (-1)^{\bar{a}\bar{c}} \sum_{e \in \B_A} (-1)^{\bar{c}\bar{e}} e \otimes ce^\vee a,
    \end{multline*}
    where, in the second equality, we used the fact that $\tr(e^\vee abc)=0$ unless $\bar{e}+\bar{a}+\bar{b}+\bar{c}=0$ to simplify the exponent of $-1$.
\end{proof}

Fix $m,n \in \N$.  For $1 \le i \le m+n$, define $p(i) \in \Z/2\Z$ by
\begin{equation}
    p(i) =
    \begin{cases}
        \bar{0} & \text{if } 1 \le i \le m, \\
        \bar{1} & \text{if } m+1 \le i \le m+n.
    \end{cases}
\end{equation}
Let $\Mat_{m|n}(A)$ be the associative superalgebra consisting of $(m+n) \times (m+n)$ matrices with entries in $A$, where multiplication is given by matrix multiplication and the $\Z/2\Z$-grading is defined as follows.  For $a \in A$ and $1 \le i,j \le m+n$, let $a_{(i,j)} \in \Mat_{m|n}(A)$ denote the matrix with $a$ in the $(i,j)$ position and $0$ in all other positions.  Then, for homogeneous $a \in A$, we define
\[
    \overline{a_{(i,j)}} = \bar{a} + p(i) + p(j).
\]
It is straightforward to verify that $\Mat_{m|n}(A)$ is a symmetric Frobenius superalgebra, with trace
\[
    \tr_{m|n} := \tr \circ \str \colon \Mat_{m|n}(A) \to \kk,
\]
where $\str$ is the usual supertrace.  Thus, for $1 \le i,j \le m+n$ and $a \in A$, we have
\[
    \tr_{m|n}(a_{(i,j)}) = \delta_{i,j} (-1)^{p(i)} \tr(a).
\]
\alex{You should verify the above statements in your thesis.}

Let $\fg = \gl_{m|n}(A)$ denote the Lie superalgebra associated to $\Mat_{m|n}(A)$.  Precisely, $\fg$ is equal to $\Mat_{m|n}(A)$ as a $\kk$-supermodule and the Lie superbracket is defined by
\[
    [M,N] = MN - (-1)^{\bar{M}\bar{N}} NM
\]
for homogeneous $M,N \in \fg$ and extended by linearity.  We have that
\begin{equation}
    \B_\fg := \{b_{(i,j)} : b \in \B_A,\ 1 \le i,j \le m+n\}
\end{equation}
is a basis for $\fg$ with left dual basis (with respect to $\tr_{m|n}$) given by
\begin{equation}
    (b_{(i,j)})^\vee = (-1)^{p(j)} b^\vee_{(j,i)}.
\end{equation}
Note that, here and in what follows, we adopt the convention that we apply the symbol $\vee$ before considering subscripts.  Thus, for example, $b^\vee_{(i,j)} = (b^\vee)_{(i,j)}$.

\begin{eg} \label{gleg}
    When $A = \kk$, then $\gl_{m|n}(\kk)$ is the usual general linear superalgebra over $\kk$.
\end{eg}

\begin{eg}
   If $A = \kk[t]/(t^l)$, $l \ge 2$, then $\gl_{m|n}(A)$ is a \emph{truncated current superalgebra}.  When $n=0$, this is also known as a \emph{Takiff algebra}.
\end{eg}

\begin{eg} \label{mirror}
    Let $\Cl$ denote the two-dimensional Clifford superalgebra generated by an odd element $c$ satisfying $c^2=1$.  Recall that the \emph{queer Lie superalgebra} $\mathfrak{q}_n(\kk)$ has even and odd parts
    \[
        \mathfrak{q}_n(\kk)_{\bar{0}} =
        \left\{
            \begin{pmatrix}
                M & 0 \\
                0 & M
            \end{pmatrix}
            : M \in \Mat_n(\kk)
        \right\},
        \quad
        \mathfrak{q}_n(\kk)_{\bar{1}} =
        \left\{
            \begin{pmatrix}
                0 & M \\
                M & 0
            \end{pmatrix}
            : M \in \Mat_n(\kk)
        \right\},
    \]
    where $\Mat_n(\kk)$ denotes the set of $n \times n$ matrices with entries in $\kk$.  Then it is straightforward to verify that we have an isomorphism of Lie superalgebras $\mathfrak{q}_{m+n}(\kk) \xrightarrow{\cong} \gl_{m|n}(\Cl)$ given by
    \[
        \begin{pmatrix}
            1_{(i,j)} & 0 \\
            0 & 1_{(i,j)}
        \end{pmatrix}
        \mapsto c^{p(i)+p(j)} 1_{(i,j)},\quad
        \begin{pmatrix}
            0 & 1_{(i,j)} \\
            1_{(i,j)} & 0
        \end{pmatrix}
        \mapsto c^{p(i)+p(j)+1} 1_{(i,j)},\quad
        1 \le i,j \le m+n.
    \]
    \details{
        Let $\vartheta \colon \mathfrak{q}_{m+n}(\kk) \xrightarrow{\cong} \gl_{m|n}(\Cl)$ denote the given map.  Define
        \[
            E_{i,j} :=
            \begin{pmatrix}
                1_{(i,j)} & 0 \\
                0 & 1_{(i,j)}
            \end{pmatrix}
            ,\quad
            E'_{i,j} :=
            \begin{pmatrix}
                0 & 1_{(i,j)} \\
                1_{(i,j)} & 0
            \end{pmatrix}.
        \]
        The map $\vartheta$ is clearly an isomorphism of vector superspaces.  For $1 \le i,j \le n$, we have
        \begin{multline*}
            [\vartheta(E_{i,j}), \vartheta(E_{k,l})]
            = [c^{p(i)+p(j)} 1_{(i,j)}, c^{p(k)+p(l)} 1_{(k,l)}]
            = \delta_{j,k} c^{p(i)+p(l)} 1_{(i,l)} - \delta_{i,l} c^{p(j)+p(k)} 1_{(k,j)}
            \\
            = \vartheta \left( \delta_{j,k} E_{i,l} - \delta_{i,l} E_{k,j} \right)
            = \vartheta \left( [E_{i,j}, E_{k,l}] \right)
        \end{multline*}
        and
        \begin{multline*}
            [\vartheta(E'_{i,j}), \vartheta(E'_{k,l})]
            = [c^{p(i)+p(j)+1} 1_{(i,j)}, c^{p(k)+p(l)+1} 1_{(k,l)}]
            = \delta_{j,k} c^{p(i)+p(l)} 1_{(i,l)} + \delta_{i,l} c^{p(j)+p(k)} 1_{(k,j)}
            \\
            = \vartheta \left( \delta_{j,k} E_{i,l} + \delta_{i,l} E_{k,j} \right)
            = \vartheta \left( [E'_{i,j}, E'_{k,l}] \right)
        \end{multline*}
        and
        \begin{multline*}
            [\vartheta(E_{i,j}), \vartheta(E'_{k,l})]
            = [c^{p(i)+p(j)} 1_{(i,j)}, c^{p(k)+p(l)+1} 1_{(k,l)}]
            = \delta_{j,k} c^{p(i)+p(l)+1} 1_{(i,l)} - \delta_{i,l} c^{p(j)+p(k)+1} 1_{(k,j)}
            \\
            = \vartheta \left( \delta_{j,k} E'_{i,l} - \delta_{i,l} E'_{k,j} \right)
            = \vartheta \left( [E_{i,j}, E'_{k,l}] \right).
        \end{multline*}
    }
    Note that $\Cl$ is not actually a \emph{symmetric} Frobenius superalgebra in the sense discussed above (where we required the trace map to be parity preserving).  However it is a Frobenius superalgebra; see \cref{marvel}.  Up to a scalar multiple, there are two choices of homogeneous trace map, one parity preserving and one parity reversing.  Under the parity-reversing trace map, $\Cl$ is symmetric, while under the parity-preserving trace map it has nontrivial Nakayama automorphism given by $\varphi(c)=-c$.
\end{eg}

Let $A^{m|n}$ denote the $\kk$-supermodule equal to $A^{m+n}$ as a $\kk$-module, with $\Z/2\Z$-grading determined by
\[
    \overline{a e_i} = \bar{a} + p(i),\quad
    a \in A,\ 1 \le i \le m+n,
\]
where $e_i$ denotes the element of $A^{m|n}$ with a $1$ in the $i$-th entry and all other entries equal to $0$.  We will also consider $A^{m|n}$ as a left $A$-supermodule with action
\[
    a(a_1,\dotsc,a_{m+n}) = (a a_1,\dotsc, a a_{m+n}).
\]

Let $\smd\fg$ denote the category of right $\fg$-supermodules.  Let $V_+ = A^{m|n}$, written as row matrices, and let $V_- = A^{m|n}$, written as column matrices.  Then $V_+$ is naturally a right $\fg$-supermodule with action given by right matrix multiplication, while $V_-$ is a right $\fg$-supermodule with action
\begin{equation}
    v \cdot M := -(-1)^{\bar{v}\bar{M}} Mv,\quad
    v \in V_-,\ M \in \fg.
\end{equation}
\details{
    For $v \in V_-$ and $M,N \in \fg$, we have
    \begin{align*}
        v \cdot [M,N]
        &= (v \cdot M) \cdot N - (-1)^{\bar{M}\bar{N}} (v \cdot N) \cdot M \\
        &= (-1)^{\bar{v}(\bar{M} + \bar{N}) + \bar{M} \bar{N}} NMv - (-1)^{\bar{v}(\bar{M} + \bar{N})} MNv \\
        &= -(-1)^{\bar{v}(\bar{M} + \bar{N})} [M,N]v \\
        &= v \cdot [M,N].
    \end{align*}
}
We define the $\kk$-bilinear form
\begin{equation}
    B \colon V_- \times V_+ \to \kk,\quad
    B(v, w) = (-1)^{\bar{v}\bar{w}} \tr(wv).
\end{equation}
We sometimes view $B$ as a $\kk$-linear map $V_- \otimes V_+ \to \kk$ in the natural way.  Then it is straightforward to verify that $B$ is a homomorphism of right $\fg$-supermodules.
\details{
    For $v,w \in V$ and $M \in \fg$, we have
    \[
        B((v \otimes w)M)
        = B \left( v \otimes wM -(-1)^{\bar{M}(\bar{v} + \bar{w})} Mv \otimes w \right)
        = (-1)^{\bar{v}(\bar{w}+\bar{M})} (\tr(wMv) -\tr(wMv))
        = 0. \qedhere
    \]
}

\begin{rem} \label{vision}
    When working with commuting actions of superalgebras, it is most natural to work with one left action and one right action, as this avoids signs arising from the actions of the two algebras commuting past one another.  Since we will want the category introduced in \cref{sec:category} to act on the left, we choose to work with \emph{right} $\fg$-supermodules.  The reader who wishes to work with left $\fg$-supermodules instead can use the standard equivalence between the categories of right and left supermodules.  Precisely, if $V$ is a right $\fg$-supermodule, then it becomes a left $\fg$-module with action given by $X \cdot v := -(-1)^{\bar{X} \bar{v}} vX$, $X \in \fg$, $v \in V$.
\end{rem}

For $a \in A$ and $1 \le i \le m+n$, let $a_{i,\pm}$ denote the element of $V_\pm$ with $a$ in the $i$-th position and $0$ in every other position.  Then
\begin{equation}
    \B_+ := \{b_{i,+} : b \in \B_A,\ 1 \le i \le m+n\}
    \qquad \text{and} \qquad
    \B_- := \{b^\vee_{i,-} : b \in \B_A,\ 1 \le i \le m+n\}
\end{equation}
are dual bases of $V_+$ and $V_-$ with respect to the bilinear form $B$.  Define
\[
    v^\vee = (-1)^{p(i)} b^\vee_{i,-} \in \B_-
    \quad \text{for }
    v = b_{i,+} \in \B_+.
\]
Thus, we have $B(v^\vee,w) = \delta_{v,w}$ for $v,w \in \B_+$.

Define
\begin{equation} \label{Omegadef}
    \Omega := \sum_{M \in \B_\fg} M \otimes M^\vee \in \fg \otimes \fg,\qquad
    \tau := \sum_{b \in \B_A} b \otimes b^\vee \in A \otimes A.
\end{equation}
The following result will be crucial in our computations of the categorical action of the affine oriented Frobenius Heisenberg category.

\begin{lem}
    For all $u,v \in V_+$, we have
    \begin{equation} \label{kitkat}
        \tau (u \otimes v) = (-1)^{\bar{u}\bar{v}} (v \otimes u) \Omega.
    \end{equation}
\end{lem}

\begin{proof}
    It suffices to prove the result for $u=a_{k,+}$ and $v=c_{l,+}$, where $a,c \in A$ and $1 \le k,l \le m+n$.  We have
    \begin{align*}
        (a_{k,+} \otimes c_{l,+}) \Omega
        &= \sum_{i,j=1}^{m+n} \sum_{b \in \B_A} (-1)^{(\bar{b} + p(i) + p(j))(\bar{c} + p(l))+p(j)} a_{k,+} b_{(i,j)} \otimes c_{l,+} b^\vee_{(j,i)} \\
        &= \sum_{b \in \B_\fg} (-1)^{(\bar{b} + p(k))(\bar{c} + p(l)) + \bar{c}p(l)} (ab)_{l,+} \otimes (cb^\vee)_{k,+} \\
        &\overset{\mathclap{\cref{adecomp}}}{=}\ \sum_{b,e \in \B_A} (-1)^{(\bar{b} + p(k))(\bar{c} + p(l)) + \bar{c}p(l)} \tr(e^\vee ab)e_{l,+} \otimes (cb^\vee)_{k,+} \\
        &\overset{\mathclap{\cref{adecomp}}}{=}\ (-1)^{(\bar{a}+p(k))(\bar{c}+p(l))} \sum_{e \in \B_A} (-1)^{\bar{e}(\bar{c} + p(l)) + \bar{c}p(l)} e_{l,+} \otimes (c e^\vee a)_{k,+} \\
        &\overset{\mathclap{\cref{adecomp}}}{=}\ (-1)^{(\bar{a}+p(k))(\bar{c}+p(l))} \sum_{b,e \in \B_A} (-1)^{\bar{e}(\bar{c} + p(l)) + \bar{c}p(l)} e_{l,+} \otimes \tr(c e^\vee b)(b^\vee a)_{k,+} \\
        &\overset{\mathclap{\cref{yield}}}{=}\ (-1)^{(\bar{a}+p(k))(\bar{c}+p(l))} \sum_{b,e \in \B_A} (-1)^{(\bar{e}+\bar{c})(\bar{c} + p(l))} \tr(e^\vee bc) e_{l,+} \otimes (b^\vee a)_{k,+} \\
        &\overset{\mathclap{\cref{adecomp}}}{=}\ (-1)^{(\bar{a}+p(k))(\bar{c}+p(l))} \sum_{b \in \B_A} (-1)^{\bar{b}(\bar{c} + p(l))} (bc)_{l,+} \otimes (b^\vee a)_{k,+} \\
        &= (-1)^{(\bar{a}+p(k))(\bar{c}+p(l))} \tau (c_{l,+} \otimes a_{k,+}). \qedhere
    \end{align*}
\end{proof}

%========================================================================
\section{Affine oriented Frobenius Brauer categories\label{sec:category}}
%========================================================================

We introduce here our main object of study, the affine oriented Frobenius Brauer category.  This is the central charge $k=0$ case of the Frobenius Heisenberg category $\Heis_k(A)$ introduced in \cite{Sav19} and further studied in \cite{BSW-foundations}.

\begin{defin}
    The \emph{oriented Frobenius Brauer category} $\OB(A)$ associated to the symmetric Frobenius superalgebra $A$ is the strict monoidal supercategory generated by objects $\uparrow$ and $\downarrow$ and morphisms
    \begin{gather*}
        \upcross \colon \uparrow \otimes \uparrow\ \to\ \uparrow \otimes \uparrow
        \ ,\quad
        \uptok \colon \uparrow\ \to\ \uparrow
        \ ,\ a \in A,
        \\
        \rightcup \colon \one \to\ \downarrow \otimes \uparrow
        , \quad
        \rightcap \colon \uparrow \otimes \downarrow\ \to \one
        , \quad
        \leftcup \colon \one \to\ \uparrow \otimes \downarrow
        , \quad
        \leftcap \colon \downarrow \otimes \uparrow\ \to \one,
    \end{gather*}
    subject to the relations
    \begin{gather} \label{toklin}
        \begin{tikzpicture}[centerzero]
            \draw[->] (0,-0.35) -- (0,0.35);
            \token{0,0}{west}{1};
        \end{tikzpicture}
        =
        \begin{tikzpicture}[centerzero]
            \draw[->] (0,-0.35) -- (0,0.35);
        \end{tikzpicture}
        ,\quad
        \lambda\
        \begin{tikzpicture}[centerzero]
            \draw[->] (0,-0.35) -- (0,0.35);
            \token{0,0}{west}{a};
        \end{tikzpicture}
        + \mu\
        \begin{tikzpicture}[centerzero]
            \draw[->] (0,-0.35) -- (0,0.35);
            \token{0,0}{west}{b};
        \end{tikzpicture}
        =
        \begin{tikzpicture}[centerzero]
            \draw[->] (0,-0.35) -- (0,0.35);
            \token{0,0}{west}{\lambda a + \mu b};
        \end{tikzpicture}
        ,\quad
        \begin{tikzpicture}[centerzero]
            \draw[->] (0,-0.35) -- (0,0.35);
            \token{0,-0.15}{east}{b};
            \token{0,0.15}{east}{a};
        \end{tikzpicture}
        =
        \begin{tikzpicture}[centerzero]
            \draw[->] (0,-0.35) -- (0,0.35);
            \token{0,0}{west}{ab};
        \end{tikzpicture}
        \ ,
        \\ \label{wreath}
        \begin{tikzpicture}[centerzero]
            \draw[->] (0.2,-0.4) to[out=135,in=down] (-0.15,0) to[out=up,in=-135] (0.2,0.4);
            \draw[->] (-0.2,-0.4) to[out=45,in=down] (0.15,0) to[out=up,in=-45] (-0.2,0.4);
        \end{tikzpicture}
        =
        \begin{tikzpicture}[centerzero]
            \draw[->] (-0.15,-0.4) -- (-0.15,0.4);
            \draw[->] (0.15,-0.4) -- (0.15,0.4);
        \end{tikzpicture}
        \ ,\quad
        \begin{tikzpicture}[centerzero]
            \draw[->] (0.3,-0.4) -- (-0.3,0.4);
            \draw[->] (0,-0.4) to[out=135,in=down] (-0.25,0) to[out=up,in=-135] (0,0.4);
            \draw[->] (-0.3,-0.4) -- (0.3,0.4);
        \end{tikzpicture}
        =
        \begin{tikzpicture}[centerzero]
            \draw[->] (0.3,-0.4) -- (-0.3,0.4);
            \draw[->] (0,-0.4) to[out=45,in=down] (0.25,0) to[out=up,in=-45] (0,0.4);
            \draw[->] (-0.3,-0.4) -- (0.3,0.4);
        \end{tikzpicture}
        \ ,\quad
        \begin{tikzpicture}[centerzero]
            \draw[->] (0.3,-0.4) -- (-0.3,0.4);
            \draw[->] (-0.3,-0.4) -- (0.3,0.4);
            \token{-0.15,-0.2}{east}{a};
        \end{tikzpicture}
        =
        \begin{tikzpicture}[centerzero]
            \draw[->] (0.3,-0.4) -- (-0.3,0.4);
            \draw[->] (-0.3,-0.4) -- (0.3,0.4);
            \token{0.15,0.2}{west}{a};
        \end{tikzpicture}
        \ ,
        \\ \label{inversion}
        \begin{tikzpicture}[centerzero]
            \draw[<-] (0.2,-0.4) to[out=135,in=down] (-0.15,0) to[out=up,in=-135] (0.2,0.4);
            \draw[->] (-0.2,-0.4) to[out=45,in=down] (0.15,0) to[out=up,in=-45] (-0.2,0.4);
        \end{tikzpicture}
        =
        \begin{tikzpicture}[centerzero]
            \draw[<-] (-0.15,-0.4) -- (-0.15,0.4);
            \draw[->] (0.15,-0.4) -- (0.15,0.4);
        \end{tikzpicture}
        \ ,\quad
        \begin{tikzpicture}[centerzero]
            \draw[->] (0.2,-0.4) to[out=135,in=down] (-0.15,0) to[out=up,in=-135] (0.2,0.4);
            \draw[<-] (-0.2,-0.4) to[out=45,in=down] (0.15,0) to[out=up,in=-45] (-0.2,0.4);
        \end{tikzpicture}
        =
        \begin{tikzpicture}[centerzero]
            \draw[->] (-0.15,-0.4) -- (-0.15,0.4);
            \draw[<-] (0.15,-0.4) -- (0.15,0.4);
        \end{tikzpicture}
        \ ,\quad
        \begin{tikzpicture}[centerzero]
            \draw[->] (0,-0.4) to[out=up,in=0] (-0.25,0.15) to[out=180,in=up] (-0.4,0) to[out=down,in=180] (-0.25,-0.15) to[out=0,in=down] (0,0.4);
        \end{tikzpicture}
        =
        \begin{tikzpicture}[centerzero]
            \draw[->] (0,-0.4) -- (0,0.4);
        \end{tikzpicture}
        =
        \begin{tikzpicture}[centerzero]
            \draw[->] (0,-0.4) to[out=up,in=180] (0.25,0.15) to[out=0,in=up] (0.4,0) to[out=down,in=0] (0.25,-0.15) to[out=180,in=down] (0,0.4);
        \end{tikzpicture}
        \ ,
        \\ \label{rightadj}
        \begin{tikzpicture}[centerzero]
            \draw[->] (-0.3,-0.4) -- (-0.3,0) arc(180:0:0.15) arc(180:360:0.15) -- (0.3,0.4);
        \end{tikzpicture}
        =
        \begin{tikzpicture}[centerzero]
            \draw[->] (0,-0.4) -- (0,0.4);
        \end{tikzpicture}
        \ ,\qquad
        \begin{tikzpicture}[centerzero]
            \draw[->] (-0.3,0.4) -- (-0.3,0) arc(180:360:0.15) arc(180:0:0.15) -- (0.3,-0.4);
        \end{tikzpicture}
        =
        \begin{tikzpicture}[centerzero]
            \draw[<-] (0,-0.4) -- (0,0.4);
        \end{tikzpicture}
        \ ,
    \end{gather}
    for all $a,b \in A$ and $\lambda,\mu \in \kk$.  In the above, the left and right crossings are defined by
    \begin{equation} \label{windmill}
        \rightcross
        :=
        \begin{tikzpicture}[centerzero]
            \draw[->] (0.2,-0.3) \braidup (-0.2,0.3);
            \draw[->] (-0.4,0.3) -- (-0.4,0.1) to[out=down,in=left] (-0.2,-0.2) to[out=right,in=left] (0.2,0.2) to[out=right,in=up] (0.4,-0.1) -- (0.4,-0.3);
        \end{tikzpicture}
        \ ,\qquad
        \leftcross
        \ :=\
        \begin{tikzpicture}[centerzero]
            \draw[->] (-0.2,-0.3) \braidup (0.2,0.3);
            \draw[->] (0.4,0.3) -- (0.4,0.1) to[out=down,in=right] (0.2,-0.2) to[out=left,in=right] (-0.2,0.2) to[out=left,in=up] (-0.4,-0.1) -- (-0.4,-0.3);
        \end{tikzpicture}
        \ .
    \end{equation}
    The parity of $\uptok$ is $\bar{a}$, and all the other generating morphisms are even.
\end{defin}

\begin{rem}
    One can define $\OB(A)$ equivalently as the strict monoidal supercategory generated by the morphisms
    \[
        \upcross,\quad \leftcross,\quad \rightcup,\quad \rightcap,\quad \uptok,\ a \in A,
    \]
    subject to the relations \cref{toklin,wreath,rightadj}, and the first two equalities in \cref{inversion}.  In this presentation, one \emph{defines} the left cup and cap by
    \[
        \leftcup
        =
        \begin{tikzpicture}[anchorbase]
            \draw[<-] (-0.15,0.3) to[out=-45,in=90] (0.15,0) arc(360:180:0.15) to[out=90,in=225] (0.15,0.3);
        \end{tikzpicture}
        \qquad \text{and} \qquad
        \leftcap
        =
        \begin{tikzpicture}[anchorbase]
            \draw[<-] (-0.15,-0.3) to[out=45,in=-90] (0.15,0) arc(0:180:0.15) to[out=-90,in=135] (0.15,-0.3);
        \end{tikzpicture}
        \ .
    \]
\end{rem}

We refer to the morphisms $\uptok$ as \emph{tokens}.  It follows from \cref{toklin,wreath} that the map
\[
    A \to \End_{\AOB(A)}(\uparrow),\quad a \mapsto \uptok,
\]
is a superalgebra homomorphism and, also using \cref{intlaw}, that
\[
    \begin{tikzpicture}[centerzero]
        \draw[->] (-0.2,-0.35) -- (-0.2,0.35);
        \token{-0.2,0.15}{east}{a};
        \draw[->] (0.2,-0.35) -- (0.2,0.35);
        \token{0.2,-0.15}{west}{b};
    \end{tikzpicture}
    =
    \begin{tikzpicture}[centerzero]
        \draw[->] (-0.2,-0.35) -- (-0.2,0.35);
        \token{-0.2,0}{east}{a};
        \draw[->] (0.2,-0.35) -- (0.2,0.35);
        \token{0.2,0}{west}{b};
    \end{tikzpicture}
    = (-1)^{\bar{a} \bar{b}}
    \begin{tikzpicture}[centerzero]
        \draw[->] (-0.2,-0.35) -- (-0.2,0.35);
        \token{-0.2,-0.15}{east}{a};
        \draw[->] (0.2,-0.35) -- (0.2,0.35);
        \token{0.2,0.15}{west}{b};
    \end{tikzpicture}
    \ ,\qquad
    \begin{tikzpicture}[centerzero]
        \draw[->] (-0.3,-0.3) -- (0.3,0.3);
        \draw[->] (0.3,-0.3) -- (-0.3,0.3);
        \token{0.15,-0.15}{west}{a};
    \end{tikzpicture}
    =
    \begin{tikzpicture}[centerzero]
        \draw[->] (-0.3,-0.3) -- (0.3,0.3);
        \draw[->] (0.3,-0.3) -- (-0.3,0.3);
        \token{-0.1,0.1}{east}{a};
    \end{tikzpicture}
    \ .
\]

Define the \emph{teleporter}
\begin{equation} \label{teleporter}
    \begin{tikzpicture}[anchorbase]
        \draw[->] (0,-0.4) --(0,0.4);
        \draw[->] (0.5,-0.4) -- (0.5,0.4);
        \teleport{0,0}{0.5,0};
    \end{tikzpicture}
    =
    \begin{tikzpicture}[anchorbase]
        \draw[->] (0,-0.4) --(0,0.4);
        \draw[->] (0.5,-0.4) -- (0.5,0.4);
        \teleport{0,0.2}{0.5,-0.2};
    \end{tikzpicture}
    =
    \begin{tikzpicture}[anchorbase]
        \draw[->] (0,-0.4) --(0,0.4);
        \draw[->] (0.5,-0.4) -- (0.5,0.4);
        \teleport{0,-0.2}{0.5,0.2};
    \end{tikzpicture}
    := \sum_{b \in \B_A}
    \begin{tikzpicture}[anchorbase]
        \draw[->] (0,-0.4) --(0,0.4);
        \draw[->] (0.5,-0.4) -- (0.5,0.4);
        \token{0,0.15}{east}{b};
        \token{0.5,-0.15}{west}{b^\vee};
    \end{tikzpicture}
    \overset{\cref{doubledual}}{=} \sum_{b \in \B_A}
    \begin{tikzpicture}[anchorbase]
        \draw[->] (0,-0.4) --(0,0.4);
        \draw[->] (0.5,-0.4) -- (0.5,0.4);
        \token{0,-0.15}{east}{b^\vee};
        \token{0.5,0.15}{west}{b};
    \end{tikzpicture}\ .
\end{equation}
We do not insist that the tokens in a teleporter \cref{teleporter} are drawn at the same horizontal level.  The convention when this is not the case is that $b$ is on the higher of the tokens and $b^\vee$ is on the lower one.  We will also draw teleporters in larger diagrams.  When doing so, we add a sign of $(-1)^{y \bar b}$ in front of the $b$ summand in \cref{teleporter}, where $y$ is the sum of the parities of all tokens in the diagram vertically between the tokens labeled $b$ and $b^\vee$.  For example,
\[
    \begin{tikzpicture}[anchorbase]
        \draw[->] (-1,-0.4) -- (-1,0.4);
        \draw[->] (-0.5,-0.4) -- (-0.5,0.4);
        \draw[->] (0,-0.4) -- (0,0.4);
        \draw[->] (0.5,-0.4) -- (0.5,0.4);
        \token{-1,0}{east}{a};
        \token{0,0.1}{west}{c};
        \teleport{-0.5,0.2}{0.5,-0.2};
    \end{tikzpicture}
    = \sum_{b \in \B_A} (-1)^{(\bar a + \bar c) \bar b}\
    \begin{tikzpicture}[anchorbase]
        \draw[->] (-1,-0.4) -- (-1,0.4);
        \draw[->] (-0.5,-0.4) -- (-0.5,0.4);
        \draw[->] (0,-0.4) -- (0,0.4);
        \draw[->] (0.5,-0.4) -- (0.5,0.4);
        \token{-1,0}{east}{a};
        \token{0,0.1}{west}{c};
        \token{-0.5,0.2}{east}{b};
        \token{0.5,-0.2}{west}{b^\vee};
    \end{tikzpicture}
    \ .
\]
This convention ensures that one can slide the endpoints of teleporters along strands:
\[
    \begin{tikzpicture}[anchorbase]
        \draw[->] (-1,-0.4) -- (-1,0.4);
        \draw[->] (-0.5,-0.4) -- (-0.5,0.4);
        \draw[->] (0,-0.4) -- (0,0.4);
        \draw[->] (0.5,-0.4) -- (0.5,0.4);
        \token{-1,0}{east}{a};
        \token{0,0.1}{west}{c};
        \teleport{-0.5,0.2}{0.5,-0.2};
    \end{tikzpicture}
    =
    \begin{tikzpicture}[anchorbase]
        \draw[->] (-1,-0.4) -- (-1,0.4);
        \draw[->] (-0.5,-0.4) -- (-0.5,0.4);
        \draw[->] (0,-0.4) -- (0,0.4);
        \draw[->] (0.5,-0.4) -- (0.5,0.4);
        \token{-1,0}{east}{a};
        \token{0,0.1}{east}{c};
        \teleport{-0.5,-0.2}{0.5,0.2};
    \end{tikzpicture}
    =
    \begin{tikzpicture}[anchorbase]
        \draw[->] (-1,-0.4) -- (-1,0.4);
        \draw[->] (-0.5,-0.4) -- (-0.5,0.4);
        \draw[->] (0,-0.4) -- (0,0.4);
        \draw[->] (0.5,-0.4) -- (0.5,0.4);
        \token{-1,0}{east}{a};
        \token{0,0.1}{west}{c};
        \teleport{-0.5,0.2}{0.5,0.2};
    \end{tikzpicture}
    =
    \begin{tikzpicture}[anchorbase]
        \draw[->] (-1,-0.4) -- (-1,0.4);
        \draw[->] (-0.5,-0.4) -- (-0.5,0.4);
        \draw[->] (0,-0.4) -- (0,0.4);
        \draw[->] (0.5,-0.4) -- (0.5,0.4);
        \token{-1,0}{east}{a};
        \token{0,0.1}{west}{c};
        \teleport{-0.5,-0.2}{0.5,-0.2};
    \end{tikzpicture}
    \ .
\]
It follows from \cref{beam} that tokens can ``teleport'' across teleporters in the sense that, for $a \in A$, we have
\begin{equation} \label{throw}
    \begin{tikzpicture}[anchorbase]
        \draw[->] (0,-0.5) --(0,0.5);
        \draw[->] (0.5,-0.5) -- (0.5,0.5);
        \token{0.5,-0.25}{west}{a};
        \teleport{0,0}{0.5,0};
    \end{tikzpicture}
    =
    \begin{tikzpicture}[anchorbase]
        \draw[->] (0,-0.5) --(0,0.5);
        \draw[->] (0.5,-0.5) -- (0.5,0.5);
        \token{0,0.25}{east}{a};
        \teleport{0,0}{0.5,0};
    \end{tikzpicture}
    \ ,\quad
    \begin{tikzpicture}[anchorbase]
        \draw[->] (0,-0.5) --(0,0.5);
        \draw[->] (0.5,-0.5) -- (0.5,0.5);
        \token{0,-0.25}{east}{a};
        \teleport{0,0}{0.5,0};
    \end{tikzpicture}
    =
    \begin{tikzpicture}[anchorbase]
        \draw[->] (0,-0.5) --(0,0.5);
        \draw[->] (0.5,-0.5) -- (0.5,0.5);
        \token{0.5,0.25}{west}{a};
        \teleport{0,0}{0.5,0};
    \end{tikzpicture}
    \ ,\quad
    \begin{tikzpicture}[anchorbase]
        \draw[->] (0,-0.5) --(0,0.5);
        \draw[<-] (0.5,-0.5) -- (0.5,0.5);
        \token{0,0.25}{east}{a};
        \teleport{0,0}{0.5,0};
    \end{tikzpicture}
    =
    \begin{tikzpicture}[anchorbase]
        \draw[->] (0,-0.5) --(0,0.5);
        \draw[<-] (0.5,-0.5) -- (0.5,0.5);
        \token{0.5,0.25}{west}{a};
        \teleport{0,0}{0.5,0};
    \end{tikzpicture}
    \ ,\quad
    \begin{tikzpicture}[anchorbase]
        \draw[<-] (0,-0.5) --(0,0.5);
        \draw[->] (0.5,-0.5) -- (0.5,0.5);
        \token{0,-0.25}{east}{a};
        \teleport{0,0}{0.5,0};
    \end{tikzpicture}
    =
    \begin{tikzpicture}[anchorbase]
        \draw[<-] (0,-0.5) --(0,0.5);
        \draw[->] (0.5,-0.5) -- (0.5,0.5);
        \token{0.5,-0.25}{west}{a};
        \teleport{0,0}{0.5,0};
    \end{tikzpicture}
    \ ,
\end{equation}
where the strings can occur anywhere in a diagram (i.e.\ they do not need to be adjacent).  The endpoints of teleporters slide through crossings and they can teleport too.  For example we have
\begin{equation} \label{laser}
    \begin{tikzpicture}[anchorbase]
        \draw[->] (-0.4,-0.4) to (0.4,0.4);
        \draw[->] (0.4,-0.4) to (-0.4,0.4);
        \draw[->] (0.8,-0.4) to (0.8,0.4);
        \teleport{-0.15,0.15}{0.8,0.15};
    \end{tikzpicture}
    =
    \begin{tikzpicture}[anchorbase]
        \draw[->] (-0.4,-0.4) to (0.4,0.4);
        \draw[->] (0.4,-0.4) to (-0.4,0.4);
        \draw[->] (0.8,-0.4) to (0.8,0.4);
        \teleport{0.2,-0.2}{0.8,-0.2};
    \end{tikzpicture}
    \ ,\qquad
    \begin{tikzpicture}[anchorbase]
        \draw[->] (-0.4,-0.5) to (-0.4,0.5);
        \draw[->] (0,-0.5) to (0,0.5);
        \draw[->] (0.4,-0.5) to (0.4,0.5);
        \teleport{-0.4,-0.25}{0,-0.25};
        \teleport{0,0}{0.4,0};
    \end{tikzpicture}
    =
    \begin{tikzpicture}[anchorbase]
        \draw[->] (-0.4,-0.5) to (-0.4,0.5);
        \draw[->] (0,-0.5) to (0,0.5);
        \draw[->] (0.4,-0.5) to (0.4,0.5);
        \teleport{-0.4,-0.2}{0.4,0.3};
        \teleport{0,-.1}{0.4,-.1};
    \end{tikzpicture}
    =
    \begin{tikzpicture}[anchorbase]
        \draw[->] (-0.4,-0.5) to (-0.4,0.5);
        \draw[->] (0,-0.5) to (0,0.5);
        \draw[->] (0.4,-0.5) to (0.4,0.5);
        \teleport{-0.4,0.15}{0.4,0.15};
        \teleport{0,-.1}{0.4,-.1};
    \end{tikzpicture}
    \ .
\end{equation}

\begin{defin}
    The \emph{affine oriented Frobenius Brauer category} $\AOB(A)$ associated to the Frobenius superalgebra $A$ is the strict monoidal supercategory obtained from $\OB(A)$ by adjoining an even generator $\updot \colon \uparrow \to \uparrow$, which we call a \emph{dot}, subject to the relations
    \begin{equation} \label{dotslide}
        \begin{tikzpicture}[centerzero]
            \draw[->] (-0.3,-0.3) -- (0.3,0.3);
            \draw[->] (0.3,-0.3) -- (-0.3,0.3);
            \singdot{-0.15,0.15};
        \end{tikzpicture}
        -
        \begin{tikzpicture}[centerzero]
            \draw[->] (-0.3,-0.3) -- (0.3,0.3);
            \draw[->] (0.3,-0.3) -- (-0.3,0.3);
            \singdot{0.15,-0.15};
        \end{tikzpicture}
        =
        \begin{tikzpicture}[centerzero]
            \draw[->] (-0.2,-0.3) -- (-0.2,0.3);
            \draw[->] (0.2,-0.3) -- (0.2,0.3);
            \teleport{-0.2,0}{0.2,0};
        \end{tikzpicture}
        \ ,\qquad
        \begin{tikzpicture}[centerzero]
            \draw[->] (0,-0.3) -- (0,0.3);
            \token{0,0.1}{east}{a};
            \singdot{0,-0.13};
        \end{tikzpicture}
        =
        \begin{tikzpicture}[centerzero]
            \draw[->] (0,-0.3) -- (0,0.3);
            \token{0,-0.13}{west}{a};
            \singdot{0,0.1};
        \end{tikzpicture}
        \ ,\quad a \in A.
    \end{equation}
\end{defin}

It follows from \cref{dotslide,wreath} that we also have
\begin{equation}
    \begin{tikzpicture}[centerzero]
        \draw[->] (-0.3,-0.3) -- (0.3,0.3);
        \draw[->] (0.3,-0.3) -- (-0.3,0.3);
        \singdot{-0.15,-0.15};
    \end{tikzpicture}
    -
    \begin{tikzpicture}[centerzero]
        \draw[->] (-0.3,-0.3) -- (0.3,0.3);
        \draw[->] (0.3,-0.3) -- (-0.3,0.3);
        \singdot{0.15,0.15};
    \end{tikzpicture}
    =
    \begin{tikzpicture}[centerzero]
        \draw[->] (-0.2,-0.3) -- (-0.2,0.3);
        \draw[->] (0.2,-0.3) -- (0.2,0.3);
        \teleport{-0.2,0}{0.2,0};
    \end{tikzpicture}
    \ .
\end{equation}
In addition, endpoints of teleporters slide through dots:
\begin{equation}\label{strawberry}
    \begin{tikzpicture}[anchorbase]
        \draw[->] (-0.2,-0.5) to (-0.2,0.5);
        \draw[->] (0.2,-0.5) to (0.2,0.5);
        \singdot{-0.2,0};
        \teleport{-0.2,-0.25}{0.2,0};
    \end{tikzpicture}
    =
    \begin{tikzpicture}[anchorbase]
        \draw[->] (-0.2,-0.5) to (-0.2,0.5);
        \draw[->] (0.2,-0.5) to (0.2,0.5);
        \singdot{-0.2,0};
        \teleport{-0.2,0.25}{0.2,0};
    \end{tikzpicture}
    \ .
\end{equation}

\begin{rem}[$\Z$-grading]
    If the symmetric Frobenius superalgebra $A$ is $\Z$-graded with the trace map $\tr$ of degree $-\dA$, then the categories $\OB(A)$ and $\AOB(A)$ are also naturally $\Z$-graded.  The $\Z$-degrees of the generating morphisms are as follows:
    \[
        \deg \left( \uptok \right) = \deg(a),\quad
        \deg \left( \updot \right) = \dA,\quad
        \deg \left( \upcross \right) = \deg \left( \rightcup \right) = \deg \left( \rightcap \right) = \deg \left( \leftcup \right) = \left( \leftcap \right) = 0.
    \]
    All of the results of the current paper can be carried out in the graded setting.
\end{rem}

\begin{eg} \label{road}
    As noted in the introduction, when $A=\kk$, the categories $\OB(\kk)$ and $\AOB(\kk)$ are the \emph{oriented Brauer} and \emph{affine oriented Brauer categories}, respectively; see \cite{BCNR17}.  The endomorphism algebras of $\OB(\kk)$ are \emph{oriented Brauer algebras}, which are isomorphic to \emph{walled Brauer algebras}.
\end{eg}

\begin{rem} \label{glass}
    The definitions of $\OB(A)$ and $\AOB(A)$ can be generalized to allow for $A$ to be a (not necessarily symmetric) Frobenius superalgebra with trace map of arbitrary parity.  The only change to the relations is that the parity of the dot is equal to the parity $\overline{\tr}$ of the trace map (i.e.\ the dot is odd if the trace map is parity reversing) and the second relation in \cref{dotslide} becomes
    \[
        \begin{tikzpicture}[centerzero]
            \draw[->] (0,-0.3) -- (0,0.3);
            \token{0,0.1}{east}{a};
            \singdot{0,-0.13};
        \end{tikzpicture}
        = (-1)^{\bar{a} \overline{\tr}}\
        \begin{tikzpicture}[centerzero]
            \draw[->] (0,-0.3) -- (0,0.3);
            \token{0,-0.13}{west}{\varphi(a)};
            \singdot{0,0.1};
        \end{tikzpicture}
        \ ,\quad a \in A,
    \]
    where $\varphi$ is the Nakayama automorphism.  (This level of generality was considered in \cite{Sav19}.)  With this modification, we can take $A$ to be the two-dimensional Clifford superalgebra $\Cl$; see \cref{mirror}.  Then $\OB(\Cl)$ and $\AOB(\Cl)$ are the oriented Brauer--Clifford and degenerate affine oriented Brauer--Clifford supercategories, respectively, introduced in \cite{BCK19}.
\end{rem}

We recall some other relations, which are proved in \cite[Th.~1.3]{Sav19}, that follow from the defining relations.  In what follows, the relations not involving dots hold in both $\OB(A)$ and $\AOB(A)$.  The relations \cref{rightadj} means that $\downarrow$ is right dual to $\uparrow$.  In fact, we also have
\begin{equation} \label{leftadj}
    \begin{tikzpicture}[centerzero]
        \draw[<-] (-0.3,-0.4) -- (-0.3,0) arc(180:0:0.15) arc(180:360:0.15) -- (0.3,0.4);
    \end{tikzpicture}
    \ =\
    \begin{tikzpicture}[centerzero]
        \draw[<-] (0,-0.4) -- (0,0.4);
    \end{tikzpicture}
    \ ,\qquad
    \begin{tikzpicture}[centerzero]
        \draw[<-] (-0.3,0.4) -- (-0.3,0) arc(180:360:0.15) arc(180:0:0.15) -- (0.3,-0.4);
    \end{tikzpicture}
    \ =\
    \begin{tikzpicture}[centerzero]
        \draw[->] (0,-0.4) -- (0,0.4);
    \end{tikzpicture}
    \ ,
\end{equation}
and so $\downarrow$ is also left dual to $\uparrow$.  Thus $\OB(A)$ and $\AOB(A)$ are \emph{rigid}.  Furthermore, we have that
\begin{equation} \label{stake}
    \downcross
    :=
    \begin{tikzpicture}[centerzero]
        \draw[<-] (0.2,-0.3) \braidup (-0.2,0.3);
        \draw[->] (-0.4,0.3) -- (-0.4,0.1) to[out=down,in=left] (-0.2,-0.2) to[out=right,in=left] (0.2,0.2) to[out=right,in=up] (0.4,-0.1) -- (0.4,-0.3);
    \end{tikzpicture}
    =
    \begin{tikzpicture}[centerzero]
        \draw[<-] (-0.2,-0.3) \braidup (0.2,0.3);
        \draw[->] (0.4,0.3) -- (0.4,0.1) to[out=down,in=right] (0.2,-0.2) to[out=left,in=right] (-0.2,0.2) to[out=left,in=up] (-0.4,-0.1) -- (-0.4,-0.3);
    \end{tikzpicture}
    \ ,\quad
    \begin{tikzpicture}[centerzero]
        \draw[<-] (0,-0.4) -- (0,0.4);
        \singdot{0,0};
    \end{tikzpicture}
    :=
    \begin{tikzpicture}[centerzero]
        \draw[->] (-0.3,0.4) -- (-0.3,-0.05) arc(180:360:0.15) -- (0,0.05) arc(180:0:0.15) -- (0.3,-0.4);
        \singdot{0,0};
    \end{tikzpicture}
    =
    \begin{tikzpicture}[centerzero]
        \draw[->] (0.3,0.4) -- (0.3,-0.05) arc(360:180:0.15) -- (0,0.05) arc(0:180:0.15) -- (-0.3,-0.4);
        \singdot{0,0};
    \end{tikzpicture}
    \ ,\quad
    \begin{tikzpicture}[centerzero]
        \draw[<-] (0,-0.4) -- (0,0.4);
        \token{0,0}{east}{a};
    \end{tikzpicture}
    :=
    \begin{tikzpicture}[centerzero]
        \draw[->] (-0.3,0.4) -- (-0.3,-0.05) arc(180:360:0.15) -- (0,0.05) arc(180:0:0.15) -- (0.3,-0.4);
        \token{0,0}{east}{a};
    \end{tikzpicture}
    =
    \begin{tikzpicture}[centerzero]
        \draw[->] (0.3,0.4) -- (0.3,-0.05) arc(360:180:0.15) -- (0,0.05) arc(0:180:0.15) -- (-0.3,-0.4);
        \token{0,0}{west}{a};
    \end{tikzpicture}
    \ ,\quad a \in A.
\end{equation}
These relations mean that dots, tokens, and crossings slide over all cups and caps in the sense that, for all orientations of the strands, we have
\begin{gather}
    \begin{tikzpicture}[anchorbase]
        \draw (-0.2,-0.2) -- (-0.2,0) arc (180:0:0.2) -- (0.2,-0.2);
        \token{-0.2,0}{east}{a};
    \end{tikzpicture}
    \ =\
    \begin{tikzpicture}[anchorbase]
        \draw (-0.2,-0.2) -- (-0.2,0) arc (180:0:0.2) -- (0.2,-0.2);
        \token{0.2,0}{west}{a};
    \end{tikzpicture}
    \ ,\quad
    \begin{tikzpicture}[anchorbase]
        \draw (-0.2,0.2) -- (-0.2,0) arc (180:360:0.2) -- (0.2,0.2);
        \token{-0.2,0}{east}{a};
    \end{tikzpicture}
    \ =\
    \begin{tikzpicture}[anchorbase]
        \draw (-0.2,0.2) -- (-0.2,0) arc (180:360:0.2) -- (0.2,0.2);
        \token{0.2,0}{west}{a};
    \end{tikzpicture}
    \ ,\quad
    \begin{tikzpicture}[anchorbase]
        \draw (-0.2,-0.2) -- (-0.2,0) arc (180:0:0.2) -- (0.2,-0.2);
        \singdot{-0.2,0};
    \end{tikzpicture}
    \ =\
    \begin{tikzpicture}[anchorbase]
        \draw (-0.2,-0.2) -- (-0.2,0) arc (180:0:0.2) -- (0.2,-0.2);
        \singdot{0.2,0};
    \end{tikzpicture}
    \ ,\quad
    \begin{tikzpicture}[anchorbase]
        \draw (-0.2,0.2) -- (-0.2,0) arc (180:360:0.2) -- (0.2,0.2);
        \singdot{-0.2,0};
    \end{tikzpicture}
    \ =\
    \begin{tikzpicture}[anchorbase]
        \draw (-0.2,0.2) -- (-0.2,0) arc (180:360:0.2) -- (0.2,0.2);
        \singdot{0.2,0};
    \end{tikzpicture}
    \ ,
    \\
    \begin{tikzpicture}[centerzero]
        \draw (-0.2,0.3) -- (-0.2,0.1) arc(180:360:0.2) -- (0.2,0.3);
        \draw (-0.3,-0.3) to[out=up,in=down] (0,0.3);
    \end{tikzpicture}
    =
    \begin{tikzpicture}[centerzero]
        \draw (-0.2,0.3) -- (-0.2,0.1) arc(180:360:0.2) -- (0.2,0.3);
        \draw (0.3,-0.3) to[out=up,in=down] (0,0.3);
    \end{tikzpicture}
    \ ,\quad
    \begin{tikzpicture}[centerzero]
        \draw (-0.2,-0.3) -- (-0.2,-0.1) arc(180:0:0.2) -- (0.2,-0.3);
        \draw (-0.3,0.3) \braiddown (0,-0.3);
    \end{tikzpicture}
    =
    \begin{tikzpicture}[centerzero]
        \draw (-0.2,-0.3) -- (-0.2,-0.1) arc(180:0:0.2) -- (0.2,-0.3);
        \draw (0.3,0.3) \braiddown (0,-0.3);
    \end{tikzpicture}
    \ .
\end{gather}
More precisely, the cups and caps equip $\OB(A)$ and $\AOB(A)$ with the structure of \emph{strict pivotal} supercategories; see \cite[(5.16)]{BSW-foundations}.  It follows from the definition of the tokens on downward strands that
\[
    \begin{tikzpicture}[centerzero]
        \draw[<-] (0,-0.35) -- (0,0.35);
        \token{0,-0.15}{east}{b};
        \token{0,0.15}{east}{a};
    \end{tikzpicture}
    = (-1)^{\bar{a}\bar{b}}\
    \begin{tikzpicture}[centerzero]
        \draw[<-] (0,-0.35) -- (0,0.35);
        \token{0,0}{west}{ba};
    \end{tikzpicture}
    \ .
\]

For $r \ge 0$, we define
\[
    \begin{tikzpicture}[centerzero]
        \draw[->] (0,-0.2) -- (0,0.2);
        \multdot{0,0}{east}{r};
    \end{tikzpicture}
    = \left( \updot \right)^{\circ r}.
\]
We adopt the following conventions for bubbles with a negative number of dots:
\begin{equation}
    \cbubble{a}{r} = - \delta_{r,-1} \tr(a),\quad
    \ccbubble{a}{r} = \delta_{r,-1} \tr(a) \quad \text{if } r < 0.
\end{equation}
For any homogeneous $a \in A$, we define
\begin{equation}\label{adag}
    a^\dagger := \sum_{b \in \B_A} (-1)^{\bar a \bar b} b a  b^\vee,
\end{equation}
which is independent of the choice of the basis $\B_A$.

Then we have the \emph{infinite Grassmannian relation}
\begin{equation} \label{infgrass}
    \sum_{r + s = n}
    \begin{tikzpicture}[centerzero]
        \bubleftblank{-0.4,0};
        \bubrightblank{0.4,0};
        \token{-0.6,0}{east}{a};
        \token{0.6,0}{west}{b};
        \multdot{-0.4,-0.2}{north}{r-1};
        \multdot{0.4,-0.2}{north}{s-1};
        \teleport{-0.2,0}{0.2,0};
    \end{tikzpicture}
    \ = - \delta_{n,0} \tr(ab) 1_\one,
\end{equation}
and the \emph{bubble slide relations}
\begin{equation} \label{bubslide}
    \begin{tikzpicture}[anchorbase]
        \draw[->] (0,-0.5) to (0,0.5);
        \bubleft{0.6,0}{a}{r};
    \end{tikzpicture}
    =
    \begin{tikzpicture}[anchorbase]
        \draw[->] (0,-0.5) to (0,0.5);
        \bubleft{-0.6,0}{a}{r};
    \end{tikzpicture}
    \ - \sum_{s,t \ge 0}
    \begin{tikzpicture}[anchorbase]
        \draw[->] (0,-0.5) -- (0,0.5);
        \draw[->] (-0.5,0.2) arc(90:450:0.2);
        \multdot{-0.7,0}{east}{r-s-t-2};
        \teleport{-0.3,0}{0,0};
        \multdot{0,-0.25}{west}{s+t};
        \token{0,0.25}{west}{a^\dagger};
    \end{tikzpicture}
    \ ,
    \begin{tikzpicture}[anchorbase]
        \draw[->] (0,-0.5) to (0,0.5);
        \bubright{-0.6,0}{a}{r};
    \end{tikzpicture}
    =
    \begin{tikzpicture}[anchorbase]
        \draw[->] (0,-0.5) to (0,0.5);
        \bubright{0.6,0}{a}{r};
    \end{tikzpicture}
    \ - \sum_{s,t \ge 0}
    \begin{tikzpicture}[anchorbase]
        \draw[->] (0,-0.5) -- (0,0.5);
        \draw[->] (0.5,0.2) arc(90:-270:0.2);
        \multdot{0.7,0}{west}{r-s-t-2};
        \teleport{0.3,0}{0,0};
        \multdot{0,-0.25}{east}{s+t};
        \token{0,0.25}{east}{a^\dagger};
    \end{tikzpicture}
    \ .
\end{equation}
The braid relation
\begin{equation}
    \begin{tikzpicture}[centerzero]
        \draw (0.3,-0.4) -- (-0.3,0.4);
        \draw (0,-0.4) \braidup (-0.3,0) \braidup (0,0.4);
        \draw (-0.3,-0.4) -- (0.3,0.4);
    \end{tikzpicture}
    =
    \begin{tikzpicture}[centerzero]
        \draw (0.3,-0.4) -- (-0.3,0.4);
        \draw (0,-0.4) \braidup (0.3,0) \braidup (0,0.4);
        \draw (-0.3,-0.4) -- (0.3,0.4);
    \end{tikzpicture}
\end{equation}
also holds for all orientations of the strands.

We next recall the basis theorem for $\AOB(A)$.  Recall that the \emph{cocenter} $C(A)$ of $A$ is the quotient of $A$ by the subspace spanned by $ab - (-1)^{\bar{a}\bar{b}} ba$ for all homogeneous $a,b \in A$.  For $a \in A$, we let $\cocenter{a}$ denote its canonical image in $C(A)$.

We define $\Sym(A)$ to be the symmetric superalgebra generated by the vector superspace $C(A)[x]$, where $x$ here is an even indeterminate.  For $n \in \Z$ and $a \in A$, let $e_n(a) \in \Sym(A)$ denote
\begin{equation}
    e_n(a) :=
    \begin{cases}
        0 & \text{if $n < 0$}, \\
        \tr(a) & \text{if $n=0$}, \\
        \cocenter{a} x^{n-1} & \text{if $n > 0$}.
    \end{cases}
\end{equation}
This defines a parity-preserving linear map $e_n \colon A \rightarrow \Sym(A)$.  By \cite[Lem.~7.1]{BSW-foundations}, for each $n \in \Z$ there is a unique parity-preserving linear map $h_n \colon A \rightarrow \Sym(A)$ such that
\begin{equation} \label{sake}
    \sum_{r+s=n} \sum_{c \in \B_A} (-1)^r e_r(ac) h_s(c^\vee b) = \delta_{n,0} \tr(ab),
    \qquad \text{for all }a,b \in A.
\end{equation}
In the special case that $A = \kk$, $\Sym(A)$ may be identified with the algebra of symmetric functions so that $e_n(1)$ corresponds to the $n$-th elementary symmetric function and $h_n(1)$ corresponds to the $n$-th complete symmetric function.

It follows from \cref{sake,infgrass} that we have an homomorphism of superalgebras
\begin{equation} \label{beta}
    \beta \colon \Sym(A) \to \End_{\AOB(A)}(\one),\quad
    e_n(a) \mapsto (-1)^{n-1} \ccbubble{a}{n-1},\quad
    h_n(a) \mapsto \cbubble{a}{n-1},\ n \ge 1.
\end{equation}
Let $X = X_n \otimes \dotsb \otimes X_1$ and $Y = Y_m \otimes \dotsb \otimes Y_1$ be objects of $\AOB(A)$ for $X_i, Y_j \in \{\uparrow, \downarrow\}$.  An \emph{$(X,Y)$-matching} is a bijection between the sets
\[
    \{i : X_i = \uparrow\} \sqcup \{j : Y_j = \downarrow\}
    \quad \text{and} \quad
    \{i : X_i = \downarrow\} \sqcup \{j : Y_j = \uparrow\}.
\]
A \emph{reduced lift} of an $(X,Y)$-matching is a string diagram representing a morphism $X \to Y$ such that
\begin{itemize}
    \item the endpoints of each string are points which correspond under the given matching;
    \item there are no floating bubbles and no dots or tokens on any string;
    \item there are no self-intersections of strings and no two strings cross each other more than once.
\end{itemize}
For each $(X,Y)$ matching, fix a set $D(X,Y)$ consisting of a choice of reduced lift for each $(X,Y)$-matching.  Then let $D_\circ(X,Y)$ denote the set of all morphisms that can be obtained from the elements of $D(X,Y)$ by adding a nonnegative number of dots and one element of $\B_A$ near to the terminus of each string (i.e.\ such that there are no crossings between the terminus and the dots and elements of $\B_A$).

Using the homomorphism $\beta$ from \cref{beta}, we have that, for $X,Y \in \AOB(A)$, $\Hom_{\AOB(A)}(X,Y)$ is a right $\Sym(A)$-supermodule under the action
\[
    \phi \theta := \phi \otimes \beta(\theta),\quad
    \phi \in \Hom_{\AOB(A)}(X,Y),\ \theta \in \Sym(A).
\]

\begin{theo}[{\cite[Th.~7.2]{BSW-foundations}}] \label{basis}
    For $X,Y \in \AOB(A)$, the morphism space $\Hom_{\AOB(A)}(X,Y)$ is a free right $\Sym(A)$-supermodule with basis $D_\circ(X,Y)$.
\end{theo}

It follows from \cref{basis} that the map \cref{beta} is an isomorphism of superalgebras.  By the $n=1$ case of \cref{infgrass}, we see that
\begin{equation} \label{soccer}
    \begin{tikzpicture}[centerzero]
        \bubrightblank{0,0};
        \token{-0.2,0}{east}{a};
    \end{tikzpicture}
    =
    \begin{tikzpicture}[centerzero]
        \bubleftblank{0,0};
        \token{0.2,0}{west}{a};
    \end{tikzpicture}
    \ ,\quad a \in A.
\end{equation}
Furthermore, it follows from \cref{bubslide} that these bubbles are strictly central:
\begin{equation}
    \begin{tikzpicture}[centerzero]
        \bubrightblank{0,0};
        \token{-0.2,0}{east}{a};
        \draw[->] (0.4,-0.3) -- (0.4,0.3);
    \end{tikzpicture}
    =
    \begin{tikzpicture}[centerzero]
        \bubrightblank{0,0};
        \token{0.2,0}{west}{a};
        \draw[->] (-0.4,-0.3) -- (-0.4,0.3);
    \end{tikzpicture}
    \ ,\quad a \in A.
\end{equation}
For any linear map $\theta \colon C(A) \to \kk$, we can define the \emph{specialized oriented Frobenius Brauer category} $\OB(A,\theta)$ by imposing on $\OB(A)$ the additional relation
\begin{equation} \label{pop}
    \begin{tikzpicture}[centerzero]
        \bubrightblank{0,0};
        \token{-0.2,0}{east}{a};
    \end{tikzpicture}
    =
    \theta(a),\quad a \in A.
\end{equation}
Similarly, we can define the \emph{specialized affine oriented Frobenius Brauer category} $\AOB(A,\theta)$ by imposing on $\AOB(A)$ the relation \cref{pop}.  We will see in \cref{soap} that, under the categorical action to be defined in \cref{sec:action}, the bubbles \cref{soccer} act by multiplication by the supertrace of the map $V_+ \to V_+$, $v \mapsto av$.  Hence these actions factor through the corresponding specialized categories.

%=============================================
\section{Categorical action\label{sec:action}}
%=============================================

We are now ready to define the action of the oriented Frobenius Brauer category and the affine oriented Frobenius Brauer category on the category of supermodules for $\fg = \gl_{m|n}(A)$.

\begin{theo} \label{tiger}
    We have a monoidal superfunctor $\psi \colon \OB(A) \to \smd\fg$ given on objects by $\uparrow\ \mapsto V_+$, $\downarrow\ \mapsto V_-$ and on morphisms by
    \begin{align*}
        \psi(\upcross) &\colon V_+ \otimes V_+ \to V_+ \otimes V_+,& v \otimes w &\mapsto (-1)^{\bar{v} \bar{w}} w \otimes v,
        \\
        \psi(\uptok) &\colon V_+ \mapsto V_+,& v &\mapsto av,
        \\
        \psi(\rightcup) &\colon \kk \mapsto V_- \otimes V_+,& 1 &\mapsto \sum_{v \in \B_+} (-1)^{\bar{v}} v^\vee \otimes v,
        \\
        \psi(\leftcup) &\colon \kk \mapsto V_+ \otimes V_-,& 1 &\mapsto \sum_{v \in \B_+} v \otimes v^\vee,
        \\
        \psi(\rightcap) &\colon V_+ \otimes V_- \to \kk,& v \otimes w &\mapsto (-1)^{\bar{v}\bar{w}} B(w \otimes v),
        \\
        \psi(\leftcap) &\colon V_- \otimes V_+ \to \kk,& v \otimes w &\mapsto B(v \otimes w).
    \end{align*}
\end{theo}

\begin{proof}
    We must show that $\psi$ respects the relations \cref{toklin,wreath,inversion,rightadj}.  We first note that, using the definitions \cref{windmill,stake}, the maps
    \[
        \psi(\rightcross) \colon V_+ \otimes V_- \to V_- \otimes V_+,\quad
        \psi(\leftcross) \colon V_- \otimes V_+ \to V_+ \otimes V_-,\quad
        \psi(\downcross) \colon V_- \otimes V_- \to V_- \otimes V_-
    \]
    are all given by $v \otimes w \mapsto (-1)^{\bar{v}\bar{w}} w \otimes v$.  Thus we have
    \[
        \psi
        \left(
            \begin{tikzpicture}[centerzero]
                \draw[->] (0,-0.4) to[out=up,in=0] (-0.25,0.15) to[out=180,in=up] (-0.4,0) to[out=down,in=180] (-0.25,-0.15) to[out=0,in=down] (0,0.4);
            \end{tikzpicture}
        \right)
        \colon v \mapsto \sum_{w \in \B_+} (-1)^{\bar{w}} w^\vee \otimes w \otimes v
        \mapsto \sum_{w \in \B_-} (-1)^{\bar{w} + \bar{v}\bar{w}} w^\vee \otimes v \otimes w
        \mapsto v,
    \]
    verifying the third equality in \cref{inversion}.  The remaining relations are similarly verified by direct computation.
\end{proof}

Note that, for $a \in A$,
\begin{equation} \label{soap}
    \psi
    \left(
        \begin{tikzpicture}[centerzero]
            \bubrightblank{0,0};
            \token{-0.2,0}{east}{a};
        \end{tikzpicture}
    \right)
    =
    \psi
    \left(
        \begin{tikzpicture}[centerzero]
            \bubleftblank{0,0};
            \token{0.2,0}{west}{a};
        \end{tikzpicture}
    \right)
    = \sum_{b \in \B_+} (-1)^{\bar{v}} B(v^\vee,av) 1_\one
\end{equation}
is multiplication by the supertrace of the map $V_+ \to V_+$, $v \mapsto av$.  In particular,
\begin{equation}
    \psi
    \left(
        \begin{tikzpicture}[centerzero]
            \bubrightblank{0,0};
        \end{tikzpicture}
    \right)
    =
    \psi
    \left(
        \begin{tikzpicture}[centerzero]
            \bubleftblank{0,0};
        \end{tikzpicture}
    \right)
    = \sdim(V_+) 1_\one,
\end{equation}
where $\sdim(V_+) = m \dim(A_{\bar{0}}) + n \dim(A_{\bar{1}}) - m \dim(A_{\bar{1}}) - n \dim(A_{\bar{0}})$ is the super dimension of $V_+$.

For a supercategory $\cC$, let $\cEnd_\kk(\cC)$ denote the strict monoidal supercategory of superfunctors and supernatural transformations.  An \emph{action} of a monoidal supercategory $\cD$ on a supercategory $\cC$ is a monoidal superfunctor $\cD \to \cEnd_\kk(\cC)$.  It follows immediately from \cref{tiger} that $\OB(A)$ acts on $\smd\fg$ by
\[
    X \mapsto \psi(X) \otimes -,\quad
    f \mapsto \psi(f) \otimes -
\]
for objects $X$ in $\OB(A)$ and morphisms $f$ in $\OB(A)$.  The following result extends this action to $\AOB(A)$.

\begin{theo} \label{lion}
    We have a monoidal superfunctor $\Psi \colon \AOB(A) \to \cEnd_\kk(\smd\fg)$ given on objects by $\uparrow\ \mapsto V_+ \otimes -$, $\downarrow\ \mapsto V_- \otimes -$ and on morphisms by
    \[
        \Psi(f) = \psi(f) \otimes -,\quad
        f \in \{ \upcross, \uptok, \rightcup, \leftcup, \rightcap, \leftcap : a \in A \},
    \]
    and $\Psi(\updot) \colon V_+ \otimes - \to V_+ \otimes -$ is the supernatural transformation with components
    \[
        \Psi(\updot)_W \colon V_+ \otimes W \to V_+ \otimes W,\quad
        v \otimes w \mapsto(v \otimes w)\Omega,
    \]
    for $W \in \smd\fg$, where $\Omega$ is the element defined in \cref{Omegadef}.
\end{theo}

\begin{proof}
    In light of \cref{tiger}, it suffices to check that $\Psi(\updot)$ is a supernatural transformation, which is straightforward, and that $\Psi$ respects the relations \cref{dotslide}.  To verify the first relation in \cref{dotslide}, we compute that
    \[
        \Psi
        \left(
            \begin{tikzpicture}[centerzero]
                \draw[->] (-0.20,-0.20) -- (0.20,0.20);
                \draw[->] (0.20,-0.20) -- (-0.20,0.20);
                \singdot{-0.08,0.08};
            \end{tikzpicture}
        \right)_W
        \colon V_+ \otimes V_+ \otimes W \to V_+ \otimes V_+ \otimes W
    \]
    is the map given by
    \begin{align*}
        u \otimes v \otimes w
        &\mapsto (-1)^{\bar{u}\bar{v}} v \otimes u \otimes w
        \\
        &\mapsto \sum_{M \in \B_\fg} (-1)^{\bar{u}\bar{v}} \left( (-1)^{\bar{u}\bar{M}} vM \otimes u M^\vee \otimes w + (-1)^{(\bar{u}+\bar{w}) \bar{M}} vM \otimes u \otimes w M^\vee \right).
    \end{align*}
    Similarly,
    \[
        \Psi
        \left(
            \begin{tikzpicture}[centerzero]
                \draw[->] (-0.20,-0.20) -- (0.20,0.20);
                \draw[->] (0.20,-0.20) -- (-0.20,0.20);
                \singdot{0.09,-0.09};
            \end{tikzpicture}
        \right)_W
        \colon V_+ \otimes V_+ \otimes W \to V_+ \otimes V_+ \otimes W
    \]
    is the map given by
    \[
        u \otimes v \otimes w
        \mapsto \sum_{M \in \B_\fg} (-1)^{\bar{w}\bar{M}} u \otimes vM \otimes wM^\vee
        \mapsto \sum_{M \in \B_\fg} (-1)^{\bar{u}(\bar{v} + \bar{M})+\bar{w}\bar{M}} vM \otimes u \otimes wM^\vee.
    \]
    Thus,
    \[
        \Psi
        \left(
            \begin{tikzpicture}[centerzero]
                \draw[->] (-0.20,-0.20) -- (0.20,0.20);
                \draw[->] (0.20,-0.20) -- (-0.20,0.20);
                \singdot{-0.08,0.08};
            \end{tikzpicture}
            -
            \begin{tikzpicture}[centerzero]
                \draw[->] (-0.20,-0.20) -- (0.20,0.20);
                \draw[->] (0.20,-0.20) -- (-0.20,0.20);
                \singdot{0.09,-0.09};
            \end{tikzpicture}
        \right)_W
        (u \otimes v \otimes w)
        = (-1)^{\bar{u}\bar{v}} (v \otimes u)\Omega \otimes w
        \overset{\cref{kitkat}}{=} \Psi
        \left(
            \begin{tikzpicture}[centerzero]
                \draw[->] (-0.2,-0.3) -- (-0.2,0.3);
                \draw[->] (0.2,-0.3) -- (0.2,0.3);
                \teleport{-0.2,0}{0.2,0};
            \end{tikzpicture}
        \right)
        (u \otimes v \otimes w).
    \]

    To verify the second relation in \cref{dotslide} we compute that, for $a \in A$, we have
    \[
        \Psi
        \left(
            \begin{tikzpicture}[centerzero]
                \draw[->] (0,-0.3) -- (0,0.3);
                \token{0,0.1}{east}{a};
                \singdot{0,-0.13};
            \end{tikzpicture}
        \right)_W
        (v \otimes w)
        = \sum_{M \in \B_\fg} (-1)^{\bar{w}\bar{M}} avM \otimes wM^\vee
        =
        \Psi
        \left(
            \begin{tikzpicture}[centerzero]
                \draw[->] (0,-0.3) -- (0,0.3);
                \token{0,-0.13}{east}{a};
                \singdot{0,0.1};
            \end{tikzpicture}
        \right)_W
        (v \otimes w). \qedhere
    \]
\end{proof}

As explained in the introduction, when $A=\kk$, \cref{tiger,lion} recover known results.  Furthermore, as noted in \cref{glass}, the definitions of $\OB(A)$ and $\AOB(A)$ can be generalized to allow $A$ to be the two-dimensional Clifford superalgebra.  In this case, the actions described in \cref{tiger,lion} correspond to those described in \cite[\S4.2 and Th.~4.4]{BCK19} on supermodules for the queer Lie superalgebra (see \cref{mirror}).

The center $Z(\cEnd(\smd\fg)) := \End_{\cEnd(\smd\fg)}(\one)$ of the category $\cEnd(\smd\fg)$ can be naturally identified with $Z(U(\fg))$ via the map
\begin{equation} \label{mango}
    \rho \colon Z(U(\fg)) \xrightarrow{\cong} Z(\cEnd(\smd\fg)),\quad
    u \mapsto \rho_u,
\end{equation}
where $\rho_u$ is the supernatural transformation whose $W$-component for $W \in \smd\fg$ is
\[
    (\rho_u)_W \colon W \to W,\quad w \mapsto (-1)^{\bar{u}\bar{w}} wu.
\]
Then it follows from \cref{lion,beta} that we have a homomorphism of superalgebras
\[
    \rho^{-1} \circ \Psi \circ \beta \colon \Sym(A) \to Z(U(\fg)).
\]
The following proposition describes this map explicitly.

\begin{prop} \label{shark}
    The element
    \[
        \rho^{-1} \circ \Psi \left( \ccbubble{a}{r} \right)
        = (-1)^r \rho^{-1} \circ \Psi \circ \beta (e_{r+1}(a))
        \in Z(U(\fg))
    \]
    is given by
    \[
        \sum_{\substack{1 \le i_1,\dotsc,i_r \le d \\ b_1,\dotsc,b_r \in \B_A}} (-1)^{a \bar{b}_r + \sum_{k=1}^r \bar{b}_k \bar{b}_{k+1}} (b_2 b_1)_{i_2,i_1} (b_3 b_2^\vee)_{i_3,i_2} \dotsm (b_r b_{r-1}^\vee)_{i_r,i_{r-1}} (b_{r+1}^\vee a b_r^\vee)_{i_{r+1},i_r},
    \]
    where we adopt the convention that $i_{r+1}=i_1$ and $b_{r+1}=b_1$.
\end{prop}

\begin{proof}
    For $W \in \smd\fg$, we compute that $\Psi \left( \ccbubble{a}{r} \right)_W$ is the map
    \begin{align*}
        &w
        \mapsto \sum_{v \in \B_+} (-1)^{\bar{v}} v^\vee \otimes v \otimes w
        \\
        &\mapsto \sum_{v \in \B_+} (-1)^{\bar{v}+\bar{a}\bar{v}} v^\vee \otimes (a v \otimes w)\Omega^r
        \\
        &= \sum_{\substack{c \in \B_A \\ 1 \le k \le d}} \sum_{\substack{1 \le i_1,\dotsc,i_r \le d \\ 1 \le j_1,\dotsc,j_r \le d \\ b_1,\dotsc,b_r \in \B_A}} (-1)^{\bar{a}\bar{c} + \bar{w} \sum_p \bar{b}_p + \sum_{p<q} \bar{b}_p \bar{b}_q} c_{k,-} \otimes ac^\vee_{k,+} (b_1)_{i_1,j_1} \dotsm (b_r)_{i_r,j_r} \otimes w (b_1^\vee)_{j_1,i_1} \dotsm (b_r^\vee)_{j_r,i_r}
        \\
        &= \sum_{\substack{1 \le i_1,\dotsc,i_r,j_r \le d \\ b_1,\dotsc,b_r,c \in \B_A}} (-1)^{\bar{a}\bar{c} + \bar{w} \sum_p \bar{b}_p + \sum_{p<q} \bar{b}_p \bar{b}_q} c_{i_1,-} \otimes (ac^\vee b_1 \dotsm b_r)_{j_r,+} \otimes w (b_1^\vee)_{i_2,i_1} \dotsm (b_{r-1}^\vee)_{i_r,i_{r-1}} (b_r^\vee)_{j_r,i_r}
        \\
        &\mapsto \sum_{\substack{1 \le i_1,\dotsc,i_r \le d \\ b_1,\dotsc,b_r,c \in \B_A}} (-1)^{\bar{a}\bar{c} + \bar{w} \sum_p \bar{b}_p + \sum_{p<q} \bar{b}_p \bar{b}_q} \tr(c ac^\vee b_1 \dotsm b_r) w (b_1^\vee)_{i_2,i_1} \dotsm (b_{r-1}^\vee)_{i_r,i_{r-1}} (b_r^\vee)_{i_1,i_r}
        \\
        &= \sum_{\substack{1 \le i_1,\dotsc,i_r \le d \\ b_1,\dotsc,b_{r-1},c \in \B_A}} (-1)^{\bar{a}\bar{c} + \bar{a} \bar{w} + \sum_{p < q < r} \bar{b}_p \bar{b}_q + \bar{a} \sum_{p=1}^{r-1} \bar{b}_p + \sum_{p=1}^{r-1} \bar{b}_p} w (b_1^\vee)_{i_2,i_1} \dotsm (b_{r-1}^\vee)_{i_r,i_{r-1}} (cac^\vee b_1 \dotsm b_{r-1})_{i_1,i_r}
        \\
        &= (-1)^{\bar{a}\bar{w}} w \sum_{\substack{1 \le i_1,\dotsc,i_r \le d \\ b_1,\dotsc,b_r \in \B_A}} (-1)^{\sum_{p < q < r} \bar{b}_p \bar{b}_q + \bar{a} \sum_p \bar{b}_p} (b_1)_{i_2,i_1} \dotsm (b_{r-1})_{i_r,i_{r-1}} (b_rab_r^\vee b_1^\vee \dotsm b_{r-1}^\vee)_{i_1,i_r}
        \\
        &\overset{\mathclap{\cref{beam}}}{=}\ (-1)^{\bar{a}\bar{w}} w \sum_{\substack{1 \le i_1,\dotsc,i_r \le d \\ b_1,\dotsc,b_r \in \B_A}} (-1)^{\sum_{k=1}^r \bar{b}_k \bar{b}_{k+1} + \bar{a} \bar{b}_{r-1}} (b_1 b_r)_{i_2,i_1} (b_2 b_1^\vee)_{i_3,i_2} \dotsm (b_{r-1} b_{r-2}^\vee)_{i_r,i_{r-1}} (b_r^\vee a b_{r-1}^\vee)_{i_1,i_r}.
    \end{align*}
    The result then follows by shifting the indices of the $b_i$ by $1$.
\end{proof}

When $A=\kk$, \cref{shark} recovers the elements described in \cite[Rem.~1.4]{BCNR17}.  For $A = \Cl$ (see \cref{glass}), these central elements were computed in \cite[Th.~4.5]{BCK19}.

%=============
% Bibliography
%=============

\bibliographystyle{alphaurl}
\bibliography{AOFBC}

\end{document}